\documentclass[a4paper,reqno]{amsart}

\usepackage[utf8]{inputenc}
\usepackage[english]{babel}
\usepackage[T1]{fontenc}
\usepackage{microtype}
\usepackage{amsmath,amssymb,amsthm}
\usepackage{thmtools,thm-restate}
\usepackage{mathtools}
\usepackage{upgreek}
\usepackage{orcidlink}
\usepackage{hyperref}
\hypersetup{hidelinks}
\usepackage[nameinlink]{cleveref}
\usepackage{suffix}
\usepackage{nccmath}
\usepackage{nicefrac}
\usepackage{enumitem}

\usepackage{yfonts}
\usepackage[hang,flushmargin]{footmisc}

\crefname{equation}{equation}{equation}

\title[{Nonautonomous systems of evolution inclusions}]{Nonautonomous  systems of evolution inclusions}

\author[B.~Aigner]{Bernhard Aigner\orcidlink{0009-0009-8252-162X}}
\address[B.~Aigner]{Institut f\"{u}r Angewandte Analysis\\
  TU Bergakademie Freiberg\\
  Pr\"{u}ferstr. 9, 09599 Freiberg\\
  Germany}
\email[B.~Aigner]{bernhard.aigner@doktorand.tu-freiberg.de}

\author[J.~Simsen]{Jacson Simsen\orcidlink{0000-0002-6683-1363}}
\address[J.~Simsen]{IMC\\
  Universidade Federal de Itajub\'a\\
  37500-903 Itajub\'a\\
  Minas Gerais\\
  Brazil}
\email[J.~Simsen]{jacson@unifei.edu.br}

\author[M.~Waurick]{Marcus Waurick\orcidlink{0000-0003-4498-3574}}
\address[M.~Waurick]{Institut f\"{u}r Angewandte Analysis\\
  TU Bergakademie Freiberg\\
  Pr\"{u}ferstr. 9, 09599 Freiberg\\
  Germany}
\email[M.~Waurick]{marcus.waurick@math.tu-freiberg.de}
\date{\today}


\DeclarePairedDelimiterX{\norm}[1]{\lVert}{\rVert}{#1}
\DeclarePairedDelimiterX{\abs}[1]{\lvert}{\rvert}{#1}
\DeclarePairedDelimiterX{\dset}[2]{\{}{\}}{#1\,\delimsize\vert\,\mathopen{} #2}
\DeclarePairedDelimiterX{\scprod}[2]{(}{)}{#1\delimsize| #2}
\DeclarePairedDelimiterX{\dualprod}[2]{\langle}{\rangle}{#1,#2}
\newcommand{\e}{\mathrm{e}} % euler number
\newcommand{\iu}{\mathrm{i}}% imaginary unit
\newcommand{\dd}{\mathrm{d}} % differential d
\newcommand{\dx}[1][x]{\,\dd{}#1}
\newcommand{\argdot}{\cdot}
\renewcommand{\Re}{\operatorname{Re}}

\renewcommand{\epsilon}{\varepsilon}

\newcommand{\R}{\mathbb{R}}
\newcommand{\C}{\mathbb{C}}
\newcommand{\N}{\mathbb{N}}

\newcommand{\dist}{\operatorname{dist}}
\newcommand{\diam}{\operatorname{diam}}
\newcommand{\Sel}{\textswab{Sel}}
\newcommand{\Leb}{\mathrm{L}}
\newcommand{\Sob}{\mathrm{H}}
\newcommand{\Hd}{d_{\mathrm{Hd}}}
\newcommand{\dom}{\operatorname{dom}}

\theoremstyle{definition}
\newtheorem{definition}{Definition}[subsection]
\newtheorem{remark}[definition]{Remark}
\newtheorem{example}[definition]{Example}
\newtheorem{assumptions}[definition]{Assumptions}
\newtheorem*{definition*}{Definition}
\newtheorem*{remark*}{Remark}

\theoremstyle{theorem}
\newtheorem{theorem}[definition]{Theorem}
\newtheorem{proposition}[definition]{Proposition}
\newtheorem{lemma}[definition]{Lemma}
\newtheorem{corollary}[definition]{Corollary}
\newtheorem*{theorem*}{Theorem}
\newtheorem*{proposition*}{Proposition}
\newtheorem*{lemma*}{Lemma}
\newtheorem*{corollary*}{Corollary}

\begin{document}

\maketitle

\vspace{-1ex}
\begin{abstract}
  We prove the existence of global solutions for some coupled systems of partially nonautonomous evolution inclusions comprised of a Cauchy problem with a compact resolvent semigroup generator and an evolution equation governed by a subdifferential of a real potential. Our system in particular includes nonautonomous generalized Schr\"odinger--Debye systems of inclusions with variable exponents, but extends to hyperbolic-parabolic systems of inclusions in particular to Maxwell--parabolic systems of inclusions. Methodologically, we extend an approach of Vrabie et al.\ to the nonautonomous case and make use of standard semigroup tools to accomodate non-parabolic behaviour of solutions paired with a new existence result for measurable selections. The combination of the latter two requires the set-valued coupling terms to be Hausdorff-continuous, to take bounded, convex and closed values, and to satisfy weak continuity with respect to one variable.
\end{abstract}

\tableofcontents

%%%%%%%%%%%%%%%%%%%%%%%%%%%%%%
%    S E C T I O N
\section{Introduction}
\label{sec:Intro}

In this article, we study coupled systems of inclusions of the form

\begin{equation}
  \label{eq:IVP}
  \left\{
  \begin{aligned}
    \tfrac{\dd}{{\dx[t]}}u + E u &\in F(u,v) &&\text{on}\ (0,T) \text{,}\\
    \tfrac{\dd}{{\dx[t]}}v + A(\argdot)v &\in G(u,v) &&\text{on}\ (0,T) \text{,}\\
    \bigl(u(0),v(0)\bigr) &= (u_{0},v_{0}) &&\text{in}\ H\times \mathcal{H} \text{,}
  \end{aligned}
  \right.
\end{equation}
where $F, G$ are multivalued maps; $-E$ is the generator of a semigroup with compact resolvent on a separable Hilbert space $H$ and for each $t>0$, $A(t)$ is an univalued (i.e., singleton-valued) maximal monotone operator in a real separable Hilbert space $\mathcal{H}$ of {\em subdifferential type}, i.e., $A(t)= \partial \phi^t$ with $\phi$ satisfying assumptions rendering the second part of the above system well-posed for fixed $g$ instead of $G(u,v)$ according to a result of Yotsutani.\footnote{See \cref{ass:A} in \cref{sec:Prelim}.}

Initial values $(u_{0},v_{0})$ are considered in $H\times \mathcal{H}$ and the right-hand sides, $F$ and $G$, are Hausdorff-continuous\footnote{See \cref{def:Hausdorff}.} as well as weakly (Hausdorff) continuous with respect to the first variable\footnote{See \cref{def:whc}.}  and assume nonempty, bounded, closed and convex values. Additionally they may satisfy a linear growth bound, see~\cref{sec:Prelim} for details. Well-posedness of problem \eqref{eq:IVP} for all $T>0$ follows from the continuity assumptions alone (see~\cref{th:LocalExist}), for global (in time) solutions (see~\cref{th:GlobalExist}) we require the growth bound.

This article is motivated by the work \cite{Simsen2024}, where an autonomous system of inclusions with a generalization of the Schr\"odinger--Debye system was considered. In contrast, in this article we consider a time-dependent operator $A$ that can be defined on Sobolev spaces with variable exponent, see \cref{subsec:SchroedingerDebye}. The main motivation for the considered model comes from the application in electrorheological fluids (see \cite{Ruzicka2011, Ruzicka2001, Ruzicka2000}). Another important application is in image processing (see \cite {ChenLevineRao2006, GuoLiuSunWu2011}). Equations with variable exponent growth conditions also occur in magnetostatics (see \cite{Cekic2012}) and capillarity phenomena (see \cite{Avci2013}). The time dependence of the maximal monotone operator introduces extra difficulties compared to the autonomous case. Additionally, we also correct some statements in \cite{Simsen2024} that mainly stem from the fact that solutions of the Schr\"odinger equation do not adhere to the concept of maximal regularity.

As we only require our generator to have compact resolvent, we use  the Yosida-approximation to get around the lack of maximal regularity by increasing the regularity of approximate solutions and prove suitable convergence to solutions. To make up for the lack of maximal regularity, the set-valued right-hand sides $F$ and $G$, which are only upper semicontinuous in parabolic-parabolic systems, have to satisfy the stronger Hausdorff-continuity type conditions instead. One of the main technical contributions of the present article is the proposed approximation scheme, which ensures the existence of approximating selections when the data is approximated as well.

Multivalued right-hand sides are useful when uncertainties or discontinuities appear in the reaction term, while coupled systems occur when different phenomena interact. In these cases, one chooses to work with differential inclusions instead of differential equations (see, for example, \cite{Aubin1984, Budyko1969, Diaz2002, Vrabie1994, Feireisl1991, Kapustyan2008, Melnik1998, Melnik2000, Tolstonogov1992} and the references therein). Such inclusions have been used in modeling in various applications.
They are featured in models of combustion processes in porous media \cite{Feireisl1991} and for the surface temperature on Earth \cite{Budyko1969, Diaz2002}. They appear in numerous applications such as the control of forest fires \cite{Bressan2007} or conduction of electrical impulses in nerve axons \cite{Terman1983, Terman1985}. In climatology, energy balance models may lead to evolution differential inclusions which involve the $p$-Laplacian \cite{Diaz1997, Diaz2006}. Finally, a degenerate parabolic-hyperbolic problem with a differential inclusion appears in a glaciology model (see \cite{Diaz1999}).

The article is structured as follows: We begin by recalling key notions and concepts, in particular regarding multivalued maps, selections and solution theory for the separated (single-valued) equations
\[
 \begin{cases}
   \tfrac{\dd}{{\dx[t]}}u + E u = f &\text{on}\ (0,T),\\
      u(0) = u_{0} &\text{in}\ H,
 \end{cases}\text{ and }\begin{cases}
   \tfrac{\dd}{{\dx[t]}}v + A(\argdot)v = g &\text{on}\ (0,T),\\
      v(0) = v_{0}& \text{in}\ \mathcal{H}.
 \end{cases}
\]
In \cref{sec:Prelim}, we address the solvability of the individual abstract PDE-problems. In \cref{sec:selmulmap}
we focus on measurable selections. In fact, we prove (\cref{th:CloseSelections} -- the main technical contribution of the paper) used for the well-posedness result later on. \Cref{sec:UniformConvergence} provides sufficient and (in the semigroup case additionally necessary) conditions for solutions of the decoupled single-valued problems to converge uniformly --- which is a requirement for local existence.  After  stating some results of general nature in \cref{sec:auxfa}, we then move on to prove local existence of solutions to problem \eqref{eq:IVP} in \cref{sec:LocalExist} and global existence in \cref{sec:GlobalExist}. We conclude with an extensive selection of examples in \cref{sec:Examples}. More specifically, we discuss a particular case of suitable coupling terms for the system \eqref{eq:IVP} and we cover the specific maximal monotone operator $A$ corresponding to the operator used in the generalized Schr\"odinger--Debye system. As for the possible choices of the operator $E$, we discuss the cases where $-E$ corresponds to the generator of the heat semigroup, of the Schr\"odinger operator, of the wave equation and of Maxwell's equations.

All considered Hilbert spaces are separable and real. For this note given any complex Hilbert space we consider its realification: By restricting the corresponding scalars to real scalars and considering only the real part of the scalar-product one obtains a real Hilbert space $H_{\mathbb{R}}$. The realification of a complex Hilbert space $H_{\C}$ is then $H_{\R}\oplus H_{\R}$ with real scalars. Note that multiplication by $\iu$ then corresponds to the operation of multiplying with $\begin{psmallmatrix} 0 & -1 \\ 1 & 0 \end{psmallmatrix}$. As a consequence, objects like $\iu \Delta$ ($\Delta$ being e.g., the usual Dirichlet--Laplacian) still make sense and inherit their properties like maximal dissipativity and norm conservation of the corresponding semigroup from the complex case. We denote by $\rightharpoonup$ weak convergence in the respective contexts. A colon after arrows defines the corresponding limit. For a metric space $(X,d)$, we write for $x\in X, r\geq 0$
\[
   B_X(x,r)\coloneqq \{ y\in X; d(x,y)<r\}\text{ and }   B_X[x,r]\coloneqq \{ y\in X; d(x,y)\leq r\}.
 \]
We omit the index $X$ if the metric is clear from the context.

%%%%%%%%%%%%%%%%%%%%%%%%%%%%%%
%    S E C T I O N
\section{Solvability of the single-valued problems}
\label{sec:Prelim}

For fixed initial values in $H$ and $\mathcal{H}$ respectively, we supplant the set-valued coupling terms $F$ and $G$ on the right-hand side of problem \eqref{eq:IVP} with measurable single-valued functions and split the problem into two parts,
\begin{enumerate}[leftmargin=4ex, label= \arabic*.]
  \item an abstract Cauchy problem with a (autonomous) semigroup generator $-E$ and
  \item an evolution equation with a (non-autonomous) maximal monotone operator $A(\argdot)$.
\end{enumerate}
Before we tackle the simultaneous problem, we shall comment on solvability of the individual parts.

\subsection{The monotone evolution problem}
Let $\mathcal{H}$ be a real and separable Hilbert space. The maximal monotone operator $A(t)$ for $t>0$ we assume to be of {\em subdifferential type}, i.e., $A(t)= \partial \phi^t$ for some real potential $\phi^{t} \colon \mathcal{H}\to \mathbb{R}\cup \{+\infty\}$. We briefly recap the notion of a subdifferential $\partial \phi^{t}$:

\begin{definition}
  \label{def:Subdifferential}
  Let $\phi\colon \mathcal{H}\to \mathbb{R}\cup \{+\infty\}$ be a lower semicontinuous convex function. The {\em effective domain} of $\phi$ is
  \begin{align*}
    \operatorname{dom}(\phi) &\coloneq \{u\in H; \phi(u)<+\infty\},\\
    \intertext{which we assume to be non-empty. For each $u\in \operatorname{dom}(\phi)$ the set}
    \partial\phi(u) &\coloneq \{ w\in H ; \forall v \!\in\! \mathcal{H}\colon \phi(v)-\phi(u)\geq \dualprod{w}{v-u}\}\\
    \intertext{is called the {\em subdifferential} of $\phi$ at $u$ and the domain of the subdifferential $\partial\phi$ is defined as}
    \operatorname{dom}(\partial\phi) &\coloneqq \{u\in \operatorname{dom}(\phi); \partial\phi(u)\neq \emptyset\}.
  \end{align*}
\end{definition}

We remark that the definition of the subdifferential $\partial\phi$ renders it a monotone relation. It is well-known that $\partial\phi$ is maximal monotone in $\mathcal{H}$, i.e.\ $\bigl(1+\lambda\partial\phi\bigr)^{-1}=\mathcal{H}$ for all $\lambda>0$ (see \cite[thm.~4]{Rockafellar1966}). We note that in general, $\partial \phi$ need not be single-valued.

We can now formally provide our assumptions on the operator $A$, which boil down to assumptions on the potential $\phi$.

\begin{assumptions}
  \label{ass:A}
  Let $T>0$. We assume that
  \begin{enumerate}[leftmargin = 5ex, label = (\roman*)]
    \item There exists  $Z \subseteq (0,T]$, a set of measure zero, s.t.\ for all $t \in (0,T]\!\setminus \!Z$, $\phi^t \colon \mathcal{H} \to \mathbb{R}\cup \{+\infty\}$ is lower semicontinuous, convex and has non-empty effective domain. Furthermore we assume
          \begin{itemize}[leftmargin = 3ex]
            \item $\exists \mathcal{D}\subseteq \mathcal{H}\text{ dense } \forall t \in (0,T]\!\setminus \!Z \colon \operatorname{dom}\bigl(\phi^{t}\bigr)=\mathcal{D} $.
            \item $\forall t \in (0,T]\!\setminus \!Z\colon \partial \phi^{t}$ is single-valued.
            \item $\forall t \in (0,T]\!\setminus \!Z \colon \partial\phi^t(0)=\{0\}\text{.}$
          \end{itemize}
    \item The potentials $\phi^{t}$ adhere to the following regularity condition:

          Let $\alpha \in [0,1]$ and set
          \begin{equation*}
            \beta \coloneq \begin{cases}
              2 & \text{if}\ 0\leq \alpha \leq \tfrac{1}{2}\text{,}\\
              \tfrac{1}{1-\alpha} &\text{if}\ \tfrac{1}{2} \leq \alpha \leq 1\text{.}
            \end{cases}
          \end{equation*}
          Then for any $n \in \N_{\geq 1}$ we assume there exist
          \begin{itemize}[leftmargin = 3ex]
            \item$K_{n} > 0$,
            \item  $g_{n}\colon [0,T] \rightarrow \mathbb{R}$ absolutely continuous satisfying $g_{n}' \in \Leb^{\beta}(0,T)$  and
            \item  $h_{n}\colon [0,T] \rightarrow \mathbb{R}$ of bounded variation s.t.
                  \begin{alignat*}{3}
                    &\forall &t \in [0,T]\!\setminus \!Z \,\,&\forall w \!\in \!\operatorname{dom}\bigl(\phi^t\bigr)\cap B[0,r]\,\,\forall s \!\in \![t,T]\!\setminus \!Z \,\, \exists \widetilde{w}\!\in \!\operatorname{dom}\bigl(\phi^s\bigr)\colon\\
                    & &\norm{\widetilde{w} - w} &\leq \abs[\big]{g_{n}(s) - g_{n}(t)}\bigl(\phi^{t}(w) + K_{n}\bigr)^{\alpha}\; \text{ and } \\
                    & &\phi^{s}(\widetilde{w}) &\leq \phi^{t}(w) + \abs[\big]{h_{n}(s) - h_{n}(t)}\bigl(\phi^{t}(w) + K_{n}\bigr)\text{.}
                  \end{alignat*}
          \end{itemize}
  \end{enumerate}
\end{assumptions}

We will now briefly recap solvability of the following evolution problem:
\begin{equation}
  \label{eq:SubdiffProblem}
  \left\{
  \begin{aligned}
    \tfrac{\dd}{\dx[t]}v(t) + \partial \phi^t v(t)& =f(t) \qquad 0 < t\leq T,\\
    v(0)&=v_{0},
  \end{aligned}
  \right.
\end{equation}
for some $v_{0}\in \mathcal{H}$. We will use the notion of {\em strong solutions} introduced by Yotsutani, see \cite[def.~2.1]{Yotsutani1979}.

% STYLE
\vspace{1cm}

\begin{definition}
  A function $v\colon [0,T] \rightarrow \mathcal{H}$ is called a {\em strong solution} of problem \eqref{eq:SubdiffProblem} on $[0,T]$ if
  \begin{enumerate}[leftmargin = 5ex, label = (\roman*)]
    \item $v \in \mathcal{C}([0,T];\mathcal{H})$.
    \item $v$ is absolutely continuous on any compact subset of $(0,T)$.
    \item for a.e.\ $t \in [0,T]$, $v(t) \in \operatorname{dom}(\phi^t)$ and $v(t)$ satisfies \eqref{eq:SubdiffProblem}.
  \end{enumerate}
\end{definition}

From the same source, specifically \cite[\S2]{Yotsutani1979}, we extract the following well-posedness result:

\begin{lemma}[Solution of the monotone part]
  \label{th:MonotoneSolution}
  Under \cref{ass:A}, for $T>0$, $v_{0} \in \mathcal{H}$ and $g \in \Leb^{2}(0,T;\mathcal{H})$, the problem
  \begin{equation*}
    \left\{
    \begin{aligned}
      \tfrac{\dd}{\dx[t]}v + A(\argdot)v &= g &&\text{on}\ (0,T) \text{,}\\
      v(0) &= v_{0} &&\text{in}\ \mathcal{H}\text{,}
    \end{aligned}
    \right.
  \end{equation*}
  where $A(t)=\partial \phi^t$, has a unique strong solution $v_{g} \in \mathcal{C}\bigl([0,T];\mathcal{H}\bigr)$. Furthermore, $v_{g}(t) \in \mathcal{D}$ for almost all $t \in [0,T]$.
\end{lemma}

\subsection{The Cauchy problem}
For the semigroup part, we additionally recap the notion of {\em extrapolation} spaces and operators:
\begin{definition}
  Let $H_{0}$ and $H_{1}$ be Hilbert spaces and let $C\colon H_{0}\supseteq \operatorname{dom}(C) \to H_{1}$ be linear, densely defined and closed. We set
  \begin{align*}
    C^{\diamond}&\colon H_{1}\to \operatorname{dom}(C)'\text{,}\\
    (C^{\diamond}\varphi)(x)&\!\coloneq \dualprod{\varphi}{C x}_{H_{1}} \qquad \varphi \in H_{1},x\in \operatorname{dom}(C)\text{.}
  \end{align*}
  We call $C_{-1}\coloneq (C^{\ast})^{\diamond}$ the {\em extrapolated operator} of $C$ and the dual space $\Sob^{-1}(C)\coloneq \bigl(\operatorname{dom}(C),\norm{\argdot}_{\mathrm{graph}(C)}\bigr)'$ the {\em extrapolation space}.
\end{definition}
For an introduction to and properties of extrapolation spaces we refer to \cite[ch.~9]{Waurick2022}, and particularly to  \cite{EngelNagel2001} in the context of $\mathrm{C}_0$-semigroups. Now we can state the following classical result:
\begin{lemma}[{{Solution of the semigroup part; \cite[sec.~4.2]{Pazy83}}}]
  \label{th:SemiGroupSolution}
  For $T_{m}>0$, $u_{0}\in H$ and $f \in \Leb^{2}(0,T_{m};H)$, the problem
  \begin{equation*}
    \left\{
    \begin{aligned}
      \tfrac{\dd}{\dx[t]}u + E u &= f &&\text{on}\ (0,T_{m}) \text{,}\\
      u(0) &= u_{0} &&\text{in}\ H\text{,}
    \end{aligned}
    \right.
  \end{equation*}
  has a unique (mild) solution $u_{f} \in \mathcal{C}\bigl([0,T_{m}];H\bigr)$ for any $u_{0}\in H$ given by
  \[
     u_f(t) = T(t)u_0 + \int_{0}^t T(t-s)f(s)\dx[t].
  \] 
\end{lemma}
This result is well-known; the estimate follows from a simple application of the Cauchy--Schwarz inequality and the contiuity can be obtain by substracting $u_f(t)-u_f(s)$ and using standard estimates as well as the strong continuity of $T$.

\begin{remark} One way to prove the representation formula via Duhamel's principle is to use the extrapolation space $\Sob^{-1}(E)$ and realise that then $u_0 \in \dom(-E_{-1})$, where $-E_{-1}$ generates the extrapolated semigroup $(T_{-1}(t))_{t\geq 0}$ (an extension of $T$) on $\Sob^{-1}(E)$, and $f\in \Leb^2\big(0,T_m;\dom(-E_{-1})\big)$; see, e.g., \cite[thm.~II.5.5]{EngelNagel2001} for extrapolation of generators and their semigroups. Hence, the argument in \cite[sec.~4.2]{Pazy83} applies to obtain the claimed variation of constants formula, i.e., the validity of Duhamel's principle.
\end{remark}

\section{Selections of multivalued maps}
\label{sec:selmulmap}

\subsection{Properties of multivalued maps}
To formulate the necessary conditions for the right-hand sides $F$ and $G$ in \eqref{eq:IVP}, we recall several key notions for multivalued maps. We write $\mathcal{P}(M)$ for the power set of a set $M$.

% % STYLE REASONS!!!
% \vspace{1cm}

\begin{definition}
  \label{def:Semicontinuity}
  Let $X$ be a Banach space and $U$ a topological space.
  \begin{itemize}[leftmargin = 5ex]
    \item A mapping $G\colon U \rightarrow \mathcal{P}(X)$ is called {\em upper semicontinuous (weakly upper semicontinuous)} at $u \in U$, if for each open (relatively weakly open) subset $D$ in $X$ satisfying $G(u) \subseteq D$, there exists a neighbourhood $V$ of $u$ s.t.\ $G(v) \subseteq D$ for each $v \in V$.\\
          If $G$ is upper semicontinuous (weakly upper semicontinuous) at each $u \in U$, then it is called {\em upper semicontinuous (weakly upper semicontinuous)} on $U$.
    \item A function $f \colon U \rightarrow \mathbb{R}\cup \{-\infty, +\infty\}$ is called {\em lower semicontinuous} at a point $u \in U$ if for every real $y < f(u)$ there exists a neighbourhood $V$ of $u$ s.t.\ $f(v)>y$ for each $v \in V$.
  \end{itemize}
\end{definition}

Clearly, each upper semicontinuous map $G\colon U \to \mathcal{P}(X)$ is weakly upper semicontinuous. Furthermore, if $G$ is a univalued map (i.e., the images of singletons are singletons), it is upper semicontinuous (weakly upper semicontinuous) on $U$ if and only if it is a continuous (weakly continuous) map on $U$ in the usual sense.

Often times, continuity properties alone do not prove to be enough for technical reasons and we will make use of the following terminology.
\begin{definition}
  \label{def:BCC}
   Let $X$ be a Banach space and $U$ a topological space. A mapping $G\colon U \rightarrow \mathcal{P}(X)$ is called {\em bcc} if for every $u\in U$, the image $G(u)$ is nonempty, bounded, closed and convex.
\end{definition}

We remark that this property is sometimes included in the notion of semicontinuity, e.g., in \cite[def.~3.1.3]{Vrabie1995}.

\begin{definition}
  \label{def:LinearGrowth}
  Let $H_{1}, H_{2}, H$ be Hilbert spaces and $F\colon H_{1} \times H_{2}\to \mathcal{P}(H)$.
  \begin{itemize}[leftmargin = 5ex]
    \item $F$ is called {\em bounded} if it maps bounded subsets of $H_{1} \times H_{2}$ into bounded subsets of  $H$.
    \item $F$ is said to satisfy a {\em linear growth bound} if
          \begin{equation*}
            \exists a,b,c \!> \!0 \,\forall (u,v) \!\in \!H_{1}\!\times \!H_{2}\, \forall f \!\in \!F(u,v)\colon \norm{f}_{H} \leq a \norm{u}_{H_{1}} + b \norm{v}_{H_{2}} + c.
          \end{equation*}
  \end{itemize}
\end{definition}

\begin{remark}
  \Cref{def:LinearGrowth} provides a traditional linear growth bound that we will use for our existence result in \cref{sec:GlobalExist}, i.e., \cref{th:GlobalExist}. We further comment that the notion of ``positive sublinearity'' used in \cite{SimsenWittbold2019}, that was introduced in \cite[def.~3.2.5]{Vrabie1995}, where the linear growth bound only has to hold outside of a ball around the origin and under certain dual pairing conditions, is too weak for our purposes and seems of pure mathematical convenience anyway, unifying a ``usual sublinear growth condition'' and a ``sign condition''. To us it is unclear when such a weakening is actually necessary or practical in applications.
\end{remark}

The overal general set-up applies to the semigroup side of problem \eqref{eq:IVP} without maximal regularity. This however was the key element why earlier approaches worked. To overcome this, we ask for stronger continuity conditions on the set-valued side. For that purpose we recall the notion of Hausdorff-distance next.

\begin{definition}
  \label{def:Hausdorff}
  For a Banach space $X$ and subsets $A,B \in \mathcal{P}(X)$ we define the {\em Hausdorff-distance} as
  \begin{equation*}
    \Hd (A,B)\coloneq \max \bigl\{\underset{a \in A\, b \in B}{\mathrm{sup}\;\mathrm{inf}} d(a,b), \underset{b \in B\, a \in A}{\mathrm{sup}\;\mathrm{inf}} d(a,b)\bigr\}.
  \end{equation*}
   The Haudorff-distance defines a semi-metric on the set of all bounded subsets of $X$. If both sets $A,B$ are (bounded and) closed then $\Hd(A,B) = 0$ implies $A=B$, see~\cite[\S28]{Hausdorff1957}.
   We further introduce
   \[
     \mathcal{BC}(X) \coloneqq \{ A\in \mathcal{P}(X); A \text{ bounded, closed, non-empty}\}.
   \]
  A bcc map $F\colon X \to \mathcal{P}(X)$ is called {\em Hausdorff-continuous}, if $F$ is continuous as a mapping from $(X,\|\cdot\|)$ to $(  \mathcal{BC}(X),\Hd)$. 
  \end{definition}

\begin{remark}
  It is easy to verify that any bcc map, which is Hausdorff-continuous is indeed upper semicontinuous.
\end{remark}

\subsection{Selections}
Since we investigate the multivalued problem \eqref{eq:IVP}, we have to introduce the key notion of {\em selections} for solvability. For the purposes of this subsection, let $X$ be a Banach space and $A \subseteq \mathbb{R}^{n}$, $n \geq 1$, a measurable subset.
\begin{definition}
  A map $F\colon A \to \mathcal{P}(X)$ is called {\em measurable} if  for each closed subset $C \subseteq X$ the preimage
  \begin{equation*}
    F^{-1}(C) = \{ y \in A ; F(y) \cap C \neq \emptyset \}
  \end{equation*}
  is Lebesgue measurable.\\
  A {\em selection} of $F\colon A \to \mathcal{P}(X)$, is a function $f\colon A \to X$ with $f(y) \in F(y)$ for almost every $y \in A$. We denote
  \begin{equation*}
    \Sel (F) = \{ f ;  f \text{ measurable selection of } F \}.
  \end{equation*}
\end{definition}
If $F$ is a univalued map, the definition of measurability is equivalent to the usual definition of a measurable function. We note that $\Sel (F)$ may be empty. Thus, we require criteria to find suitable selections, which we present next.
\begin{lemma}[Kuratowski--Ryll--Nardzewski, {\cite[thm.~3.1.1]{Vrabie1995}}]
  \label{th:SelectionBySeparability}
  If $X$ is separable and $F\colon A \rightarrow \mathcal{P}(X)$ is measurable s.t.\ for a.e. $y \in A$,  $F(y)$ is nonempty and closed, then  $\Sel(F) \neq \emptyset$.
\end{lemma}
If $U$, $X$, $M$ are sets and $u\colon M\to U$, $F\colon U\to \mathcal{P}(X)$, we abbreviate 
\[
   F(u)\coloneqq F\circ u \colon M\to \mathcal{P}(X).
\]

\begin{theorem}[{\cite[thm.~3.1.2]{Vrabie1995}}]
  \label{th:SelectionByWUS}
  Let $U$ be a topological space and let $\emptyset\neq M\subseteq \mathbb{R}^d$ be bounded and Lebesgue-measurable, $d \geq 1$. Let $F\colon U \to \mathcal{P}(X)$ be weakly upper semicontinuous. Assume  $u,u_n\colon M \to U$ and $f_n \in \Sel F(u_n) $ for $n \in \mathbb{N}$ satisfy
  \begin{equation*}
    f_n \rightharpoonup\colon f \quad \text{in}\ \Leb^1(M;X) \qquad \text{and} \qquad
    u_n(y) \to u(y) \quad \text{a.e.\ in}\ M\text{.}
  \end{equation*}
  Then $f \in \Sel F(u)$.
\end{theorem}

\subsection{Approximation of selections}
In this part we provide a result regarding availability of selections close enough in $\Leb^{2}$-norm, which --- to the best of our knowledge --- is new. The result reads as follows.

\begin{theorem}
  \label{th:CloseSelections}
  Let $H$ be a separable Hilbert space. Let $F\colon H\to \mathcal{P}(H)$ be bcc and Hausdorff-continuous. Let $(u_{n})_{n}$  in $\mathcal{C}\bigl([0,T];H\bigr)$ converge to $u \in  \mathcal{C}\bigl([0,T];H\bigr)$. Then
  \begin{equation*}
    \forall \epsilon \!> \!0 \,\exists n_{0}\!\in \!\N\,\forall n\!\geq \!n_{0},  f \in \Sel F(u) \,\exists f_{n}\!\in \!\Sel F(u_{n})\colon \norm{f_{n}-f}_{\Leb^{2}(0,T;H)}\leq \epsilon.
  \end{equation*}
\end{theorem}

The proof of \cref{th:CloseSelections} requires a lot of preparation. Partly this is due to the fact that Hausdorff-continuity in inifinite-dimensional contexts appears to be a non-standard setting. The reason for that might be that often times compactness is used as a standard assumption for the sets involved, which we choose to avoid here. Yet, the Hilbert space structure helps a lot in overcoming potential difficulties. For convenience of the reader, we provide a rather detailed presentation of the subtleties in this context. 

In the following remark, we recall some standard properties of the metric projection in Hilbert spaces.
\begin{remark}
  \label{rem:metrproj}
  Let $H$ be a Hilbert space and $C\subseteq H$ be non-empty, closed and convex.
  \begin{enumerate}[label = (\roman*), leftmargin = 5ex]
    \item The {\em metric projection} given by
          \begin{equation*}
            p_C \colon H \to C, \quad x \mapsto \mathrm{arg min}_{c \in C}\norm{c - x}_{H}
          \end{equation*}
          is well-defined and Lipschitz-continuous with Lipschitz-constant $1$, see \cite[p.~142]{BK15}.
    \item\label{it:ProjCrit} For all $x\in H$ and $y \in H$ we have $y = p_C(x)$ if and only if (again see  \cite{BK15})
          \[
          \forall r \in C\colon \dualprod{x - y}{r - y} \leq 0.
          \]
    \item\label{it:ProjEst} If $D\subseteq H$ is non-empty, bounded, closed, and convex, then for $x\in H$
          \[
          \norm{p_C(x)-p_{D}(x)}  \leq \sqrt{(4 \|x\|+2R)\Hd(C,D)},
          \]
          where $R>0$ is a common bound for both $C$ and $D$.\\
          Indeed, let $\varepsilon > \Hd(C,D)$, $x\in H$. Then we find $p \in C$ and $q\in D$ s.t.
          \[
          \max\{\|q - p_C(x)\|,\|p-p_{D}(x)\|\}<\varepsilon.
          \]
          \begin{align*}
            % \qquad is STYLE!!!
            \qquad\norm[\big]{p_{C}(x) \!- \!p_{D}(x)}^{2}
            & = \dualprod[\big]{x \!- \!p_D(x) - \big(x \!- \!p_C(x)\big)}{p_C(x) \!- \!p_D(x)} \\
            & = \dualprod[\big]{x \!- \!p_D(x)}{p_C(x) \!- \!p_D(x)} + \dualprod[\big]{x \!- \!p_C(x)}{p_D(x) \!- \!p_C(x)} \\
            & = \dualprod[\big]{x \!- \!p_D(x)}{p_C(x) \!- \!q} + \dualprod[\big]{x \!- \!p_D(x)}{q \!- \!p_D(x)} \\
            & \quad+  \dualprod[\big]{x \!- \!p_C(x)}{p_D(x) \!- \!p} + \dualprod[\big]{x \!- \!p_C(x)}{p \!- \!p_C(x)} \\
            & \smashoperator[l]{\mathop{\leq}^{\ref{it:ProjCrit}}_{\text{Cauchy--Schwarz}}} \norm[\big]{x \!- \!p_D(x)} \varepsilon + \norm[\big]{x \!- \!p_C(x)} \varepsilon.
          \end{align*}
          Finally, from
          \begin{align*}
            \norm{x-p_D(x)} &\leq \|x\|+\|p_D(0)\|+ \|p_D(x)-p_D(0)\| \leq 2\|x\|+\|p_D(0)\|,\\
            \intertext{and similarly,}
            \norm{x-p_C(x)} &\leq 2\|x\|+R,
          \end{align*}
          it follows that
          \begin{equation*}
            \norm{p_{C}(x)-p_{D}(x)} \leq \sqrt{(4 \|x\|+2R)\Hd(C,D)}.
          \end{equation*}
  \end{enumerate}
\end{remark}

Next, we recall a well-known fact on the characterisation of compact subsets of metric spaces.

\begin{remark}\label{rem:preccom} Let $X$ be a metric space. Then $X$ is compact, if and only if, $X$ is complete and {\em pre-compact} (i.e., {\em totally bounded}: for all $\varepsilon>0$ there is a finite set $G\subseteq X$ such that $X= \bigcup_{r\in G}B(r,\varepsilon)$). Note that, as $X$ is metric, $X$ is compact if and only if it is sequentially compact. This observation serves as a sensible step to show the equivalence.
\end{remark}

\begin{lemma}\label{lem:ubauto} Let $M$ be a set, $X$ a complete metric space, $\phi\colon M\to X$ and $\phi_n\colon M\to X$ s.t.\ $\phi_n[M]$ is pre-compact for all $n\in \N$. If $\phi_n\to \phi$ uniformly, then
\[
    \overline{\bigcup_{n\in \N} \phi_n[M]}\subseteq X\text{ is compact.}
\]
\end{lemma}
\begin{proof} Before we come to the proof, we require the following observation \ref{it:Balls}.
  \begin{enumerate}[label = (\roman*), leftmargin = 5ex]
          \item\label{it:Balls} Let $\varepsilon>0$. Let $\psi_1,\psi_2\colon M\to X$ and $G\subseteq X$ with $ \psi_1[M]\subseteq \bigcup_{r\in G} B(r,\varepsilon)$. If $\sup_{x\in X} d_X\bigl(\psi_1(x),\psi_2(x)\bigr)\leq \varepsilon$, then
          \[
          \psi_2[M]\subseteq \bigcup_{r\in G} B(r,2\varepsilon).
          \]
    \item Next, we claim that $\phi[M]$ is pre-compact.

          Let $\varepsilon>0$. Then uniform convergence yields $n\in \N$ s.t.\ $d_X\bigl(\phi(x),\phi_n(x)\bigr)< \varepsilon$ for all $x\in M$. Since $\phi_n[M]$ is pre-compact, we find $G\subseteq X$ finite with $ \phi_n[M]\subseteq \bigcup_{r\in G} B(r,\varepsilon)$. Hence, by our previous observation applied to $\psi_1=\phi_n$ and $\psi_2=\phi$,
\[
\phi[M]\subseteq \bigcup_{r\in G} B(r,2\varepsilon)
\]
and $\phi[M]$ is pre-compact.

\item Next, we show the actual claim. Note that both the finite union of pre-compact sets and the closure of a pre-compact set is pre-compact again. Hence, by \cref{rem:preccom} and since closed subsets of complete metric spaces are complete, it suffices to show that $\bigcup_{n\geq n_0} \phi_n[M]$ is pre-compact for some $n_0\in \N$. For this let $\varepsilon>0$. Since $\phi[M]\subseteq X$ is pre-compact, we find $G\subseteq X$ finite with $\phi[M]\subseteq \bigcup_{r\in G} B(r,\varepsilon)$.  By uniform convergence, we find $n_0\in \N$ s.t.\ for all $n\geq n_0$ and $x\in M$ we have
\[
 d_{X}\bigl(\phi(x),\phi_n(x)\bigr)<\varepsilon.
\]
By observation \ref{it:Balls}, applied to $\psi_1=\phi$ and $\psi_2=\phi_n$ for any $n\geq n_0$, we deduce
\[
\phi_n[M]\subseteq \bigcup_{r\in G} B(r,2\varepsilon); \text{ hence, } \bigcup_{n\geq n_0} \phi_n[M]\subseteq \bigcup_{r\in G} B(r,2\varepsilon)
\]and the claim follows.\qedhere
\end{enumerate}
\end{proof}

\begin{lemma}\label{lem:easycons}
Let  $X$ be a compact topological space, $Y,Z$ be metric spaces, $Y$ complete and let $(u_n)_n \in \mathcal{C}(X;Y)^{\N}$ converge to $u\in \mathcal{C}(X;Y)$, where the metric in $\mathcal{C}(X;Y)$ is given by
\[
    d_{\mathcal{C}(X;Y)} (v,w)\coloneqq \sup_{x\in X} d_Y\bigl(v(x),w(x)\bigr).
\]
If $F\colon Y\to Z$ is continuous, then $F\circ u_n \to F\circ u$ in $\mathcal{C}(X;Z)$.
\end{lemma}
\begin{proof}
  We proceed in two steps:
  \begin{enumerate}[label = (\roman*), leftmargin = 5ex]
    \item\label{it:FirstStep} At first we show the claim under the additional assumption that $F$ is uniformly continuous. Let $\varepsilon>0$. Then we find $\delta>0$ such that whenever $y_1,y_2\in Y$ satisfy $d_Y(y_1,y_2)\leq \delta$, we infer $d_Z\big(F(y_1),F(y_2)\big)\leq \varepsilon$. By assumption, there is $n_0\in \N$ s.t.\ for all $n\geq n_0$, we have
\[
   d_{\mathcal{C}(X;Y)} (u_n,u) = \sup_{x\in X} d_Y\bigl(u_n(x),u(x)\bigr)\leq \delta.
\]
Hence, for $n\geq n_0$
\[
    d_{\mathcal{C}(X;Z)} (F\circ u_n,F\circ u)=\sup_{x\in X} d_Z\bigl(F(u_n(x)),F(u(x))\bigr)\leq \varepsilon.
\]
\item For the general case, as $X$ is compact, it follows that $u_n[X]\subseteq Y$ is compact. In particular, it is pre-compact by \cref{rem:preccom} for all $n\in \N$. Since $u_n\to u$ uniformly and $Y$ is complete, it follows that
\[
   \widetilde{Y}\coloneqq \overline{\bigcup_{n\in \N}u_n[X]} \subseteq Y
\]is compact. Note in passing that $u[X]\subseteq \widetilde{Y}$. Finally, $\widetilde{F}\coloneqq F|_{\widetilde{Y}}$ is uniformly continuous as it is a continuous mapping of metric spaces defined on a compact set. Moreover, for all $n\in \N$, $F\circ u_n = \widetilde{F}\circ u_n$ and $F\circ u=\widetilde{F}\circ u$, by construction. Hence, \ref{it:FirstStep} applies to $\widetilde{F}\circ u_n$ and the claim follows.\qedhere
\end{enumerate}
\end{proof}

Next, we provide some results anticipating the subtleties of discontinuity of the intersection for Hausdorff-continuity.

\begin{lemma}\label{lem:chinnonempty} Let $F\colon H\to \mathcal{P}(H)$ be bcc and Hausdorff-continuous, $y\in H$, $\varepsilon>0$. Then there is $\delta>0$ s.t.\ for all $z\in B(y,\delta)$ and $x\in F(y)$:
\[
     B(x,\varepsilon)\cap F(z) \neq \emptyset.
\]
\end{lemma}
\begin{proof}
Since $F$ is Hausdorff-continuous, we find $\delta>0$ s.t.\ for all $z\in B(y,\delta)$, we have $\Hd(F(y),F(z))< \nicefrac{\varepsilon}{2}$. Let $x\in F(y)$. Then $\dist\bigl(x,F(z)\bigr)\leq  \Hd\bigl(F(y),F(z)\bigr)$ and we find  $x' \in F(z)$ such that
\[
   \|x'-x\|\leq \nicefrac{\varepsilon}{2}.
\]
In particular, $x' \in B(x,\varepsilon)\cap F(z)$.
\end{proof}
The next result is an infinite-dimensional analogue of linear regularity for convex sets, see, e.g., \cite{BB96}. 
\begin{lemma}\label{lem:slaterinequ} Let $H$ be a Hilbert space, $A,B\subseteq H$ closed, convex and bounded. Assume there exists $x_0\in A\cap B$ and $\rho>0$ s.t.\ $B[x_0,\rho]\subseteq B$.
Then for all $x\in H$
\[
   \dist(x,A\cap B)\leq\big(1+ \tfrac{d}{\rho}\big) \big(\dist(x,A)+\dist(x,B)\big),
\]
where $d\coloneqq \diam(A\cup B)$.
\end{lemma}
\begin{proof}
Let $x\in H$ and define $a\coloneqq p_A(x)$, the metric projection of $x$ onto $A$. It follows that $\dist(x,A)=\|x-a\|$. For $\lambda\in [0,1]$ define
\[
   z_\lambda \coloneqq (1-\lambda )a+ \lambda x_0 \in A,
\]
by convexity of $A$ and introduce
\[
   B_\lambda \coloneqq (1-\lambda) B + \lambda B[x_0,\rho].
\]
Since $B[x_0,\rho]\subseteq B$ and $B$ is convex, we infer $B_\lambda\subseteq B$. Next,
we compute
\begin{align*}
  \dist(z_\lambda,B) & \leq \dist(z_\lambda,B_\lambda) \\
  & = \inf_{b\in B, u\in B[0,1]} \norm[\big]{z_\lambda - \big((1-\lambda)b +\lambda (x_0 +\rho u)\big)} \\
    & = \inf_{b\in B, u\in B[0,1]} \norm[\big]{(1-\lambda )a+ \lambda x_0 - \big((1-\lambda)b +\lambda (x_0 +\rho u)\big)} \\
    & = \inf_{b\in B, u\in B[0,1]} \| (1-\lambda )(a-b) - \lambda \rho u\|.
\end{align*}
For fixed $b\in B$, if $\norm{(1-\lambda)(a-b)}\leq \lambda \rho$, then
\[
   \inf_{u\in B[0,1]} \norm[\big]{(1-\lambda )(a-b) - \lambda \rho u} =0.
\] If $\norm{(1-\lambda)(a-b)}> \lambda \rho$, then solving the minimisation problem for fixed $b$ in $u$ yields
\begin{align*}
 \inf_{u\in B[0,1]} \norm[\big]{(1-\lambda )(a-b) - \lambda \rho u} & = \norm[\Big]{ (1-\lambda )(a-b) - \lambda \rho \tfrac{(a-b)}{\|a-b\|}} \\
 & =  (1-\lambda )\|a-b\| - \lambda \rho.
\end{align*}
Thus, in any case,
\[
  \dist(z_\lambda,B) \leq \inf_{b\in B}\max\big\{0,  (1-\lambda )\|a-b\| - \lambda \rho\big\} = \max\big\{0,(1-\lambda)\dist(a,B)-\lambda\rho\big\}.
\]
Choose $0<\lambda= \frac{\dist(a,B)}{\dist(a,B)+\rho}<1$. Then $z_\lambda \in B$. Hence, we compute
\begin{align*}
   \dist(x,A\cap B) & \leq \|x-z_\lambda\|\leq \|x-a\|+ \|a-z_\lambda\| \\
                    & = \dist(x,A)+\lambda\|a-x_0\| \\
                    & = \dist(x,A)+\frac{\dist(a,B)}{\dist(a,B)+\rho}\|a-x_0\| \\
                    & \leq \dist(x,A)+\frac{\dist(a,B)}{\dist(a,B)+\rho}d; \\
\intertext{with $\|a-x\|=\dist(x,A)$ and, thus, $\dist(a,B)\leq \dist(x,A)+\dist(x,B)$ we obtain}
                    & \leq \dist(x,A)+\frac{\dist(x,A)+\dist(x,B)}{\dist(a,B)+\rho}d \\
                    & \leq  \big(1+\tfrac{d}{\rho}\big)\big(\dist(x,A)+\dist(x,B)\big).\qedhere\end{align*}
\end{proof}
In certain situations, the intersection is continuous w.r.t.\ Hausdorff-convergence, as the next result confirms.
\begin{lemma}\label{lem:hdcont} Let $H$ be a Hilbert space, $(c_n)_n$ a convergent sequence in $H$, $c\coloneqq \lim c_n$, and $(B_n)_n$ a sequence of closed, bounded, convex subsets of $H$ converging to some bounded, closed, convex, non-empty $B\subseteq H$ w.r.t.\ $\Hd$, $r>0$. If $B(c,r)\cap B\neq \emptyset$, then $B[c_n,r]\cap B_n\to B[c,r]\cap B$ w.r.t.~$\Hd$.
\end{lemma}
\begin{proof} We put $A_n\coloneqq B[c_n,r]$ and $A\coloneqq B[c,r]$.
 Denote $d\coloneqq  \sup_n (\diam(A_n \cup B_n))$. Since $B$ is bounded and $B_n\to B$ in the Hausdorff-metric, we find $\varepsilon>0$ s.t.\ $B_n\subseteq B + \varepsilon B[0,1]$ for eventually all $n\in \N$.  Thus, by boundedness of each individual $B_n$ and from
 \[A_n \subseteq \bigcup_{x\in H, \|x\|\leq \sup_{n\in \N} \|c_n\|} B[x,r]\subseteq B[0, r+ \sup_{n\in \N}\|c_n\|]
 \] for all $n\in \N$, we infer $d<\infty$. Note that we also showed that we find $R>0$ such that $\bigcup_n A_n\cup B_n\cup A\cup B\subseteq B[0,R]$.\\
 We define $K_n\coloneqq A_n\cap B_n$ for all $n\in \N$ and $K\coloneqq A\cap B$.
 \begin{enumerate}[label = (\roman*), leftmargin = 5ex]
   \item At first, we show that $\sup_{x_n \in K_n} \dist(x_n, K) \to 0$ as $n\to \infty$.

         For $x_n \in K_n = A_n \cap B_n$ we estimate\begin{align*}
   \dist(x_n,A)
   & = \inf_{u\in B[0,1]} \norm[\big]{x_n - (c + r u)}
    = \max\big\{ 0, \|x_n-c\|-r\big\} \\
   & \leq \max\big\{0, \|x_n - c_n\| + \|c_n - c\| - r\big\} \\
   & \leq \|c_n-c\|
\end{align*}
and $\dist(x_n,B) \leq \Hd(B_n,B)$.
Thus, by \cref{lem:slaterinequ}, we infer
\[
   \dist(x_n,K)\leq\big(1+ \tfrac{d}{r}\big)\big( \|c_n-c\|+\Hd(B_n,B)\big) \xrightarrow[n\to \infty]{} 0.
\]
   \item It remains to prove $\sup_{x \in K} \dist(x, K_n) \to 0$ as $n\to\infty$.

         Let $x_0 \in B(c,r)\cap B$ and $\rho>0$ be s.t.\ $B(x_0,\rho)\subseteq B(c,r)$. Next, let $\lambda \in (0,1]$ and define for $x\in K=A\cap B$
\[
 x_\lambda = (1-\lambda)x + \lambda x_0 \in B,
\]
by convexity of $B$.
From $\|x-c\|\leq r$ and
\begin{equation*}
  \|c-x_0\| +\rho=\norm[\Big]{c-x_0 + \tfrac{\rho(c-x_0)}{\|c-x_0\|}}\smashoperator[r]{\mathop{\leq}}_{B(x_0,\rho)\subseteq B(c,r)} r
\end{equation*}
we infer $\|x_0-c\|\leq r-\rho$. Consequently,
\[
  \|x_\lambda - c\|\leq (1-\lambda) r +\lambda r -\lambda \rho = r-\lambda \rho.
\]
For $n\in \N$, by convexity and closedness of $B_n$, we find $z_n(x)\in B_n$ s.t.
\[
  \| z_n(x) - x_\lambda \| \leq \Hd(B_n,B).
\]
Then
\begin{align*}
  \|z_n(x)-c_n\| & \leq \|z_n(x)-x_\lambda \|+ \|x_\lambda - c\|+ \|c-c_n\| \\
  & \leq \Hd(B_n,B)+ r-\lambda \rho+ \|c-c_n\| .
\end{align*}
We find $n_0\in \N$ s.t.\ $ \Hd(B_n,B)+ \|c-c_n\|\leq \lambda\rho$ for all $n\geq n_0$. Thus, for $n\geq n_0$, $ \|z_n(x)-c_n\| \leq r$ and, hence, $z_n(x)\in K_n$.

\noindent Next, as
\[
   \quad\dist(x,K_n)\leq \|x \!- \!z_n(x)\|\leq \|x \!- \!x_\lambda \| + \|x_\lambda \!- \!z_n(x)\| \leq 2\lambda R + \Hd(B_n,B),
\]
we infer
\[
   \limsup_{n\to\infty} \sup_{x\in K}  \dist(x,K_n) \leq 2\lambda R.
\]As $\lambda\in (0,1]$ was arbitrary, we deduce $\sup_{x\in K}  \dist(x,K_n)\to 0$, which proves the desired Hausdorff-convergence.\qedhere
 \end{enumerate}
\end{proof}
Before we finally turn to the proof of our main technical result on approximation of selections we recall Lusin's theorem.
\begin{theorem}[Lusin's theorem]
  \label{th:Lusin}
  Let $U\subseteq \R^d$, $d\in \N_{\geq1}$, be compact and endowed with the measure $\mu$ induced by the Lebesgue measure on $\R^{d}$. Let $X$ be a separable metric space and $f\colon U\to X$ measurable. Then
  \begin{equation*}
    \forall \epsilon \!> \!0 \,\,\exists K \!\subseteq \!U \,\text{compact}\colon \mu \bigl(U\! \setminus \! K\bigr) \!< \!\epsilon \quad\text{and}\quad f\vert_{K}\,\text{is continuous}.
  \end{equation*}
\end{theorem}
For a proof of Lusin's theorem for functions $f\colon [0,1]\to \mathbb{R}$ we refer to \cite[thm.~7.4.3]{Cohn1993}. The generalization is proven by standard tools. In our case, we will use $U=[0,1]$, and $X=H$ a Hilbert space. In this case, the corresponding result can be obtained from \cite[thm.~7.4.3]{Cohn1993} using an orthonormal basis. For this, note that, by Pettis' theorem (see, e.g., \cite[thm.~3.1.10]{Waurick2022}), a measurable $f\colon [0,1]\to H$ is characterised by being a.e.~separably valued (i.e., there is a set of measure zero $Z\subseteq [0,1]$ so that $f\big[[0,1]\!\setminus \!Z\big]$ is separable) and weakly measurable (i.e., for all $x'\in H$, $\langle f(\cdot),x'\rangle$ is measurable; hence admitting a Borel measurable representative).
\begin{proof}[Proof of \cref{th:CloseSelections}] Note that all mentioned sets of selections are nonempty appealing to \cref{th:SelectionBySeparability}: First observe that $\Leb^{2}(0,T;H)$ is separable since $H$ is and that Hausdorff-continuity of $F$ implies upper semicontinuity, hence implies measurability. Together these facts assure applicability of \cref{th:SelectionBySeparability}.
  
  Let $\varepsilon>0$. 
 As $u_n \to u$ uniformly (i.e., in $\mathcal{C}\bigl([0,T];H\bigr)$), by \cref{lem:easycons}, we deduce $F\circ u_n\to F\circ u$ uniformly (where $F$ takes values in $\bigl(\mathcal{BC}(H),\Hd\bigr)$). By \cref{lem:chinnonempty} (applied to $F$ and $y=u(t)$) and compactness of the closure of the union of the images of $u_n$ (see \cref{lem:ubauto}), we find $n_0\in \N$ s.t.\ for all $t\in [0,T]$ and $f\in \Sel F(u)$,  $B\big(f(t),\varepsilon\big)\cap F\big(u_n(t)\big)\neq \emptyset$ for all $n\geq n_0$. Hence, for all $t\in [0,T]$, $f\in \Sel F(u)$,
  \begin{equation*}
    \chi_{n}(t) \coloneq \bigl\{x \in F\big(u_{n}(t)\big); \norm{x - f(t)}_{H}\leq \epsilon\bigr\}
  \end{equation*}
becomes non-empty eventually for $n$ large enough (uniformly in $f$ and $t\in [0,T]$). Let $f\in \Sel F(u)$ and define $ \chi_{n}(t) $ as above.
  
  Because $F$ is bcc, the sets $\chi_{n}(t)$ are convex (as intersection of two convex sets) and (weakly) closed by Mazur's lemma, because both individual sets are strongly closed.
  
  Since $H$ is a Hilbert space, we can define the pointwise bestapproximation
  \begin{equation*}
    \pi_{n}\colon [0,T] \ni t\mapsto \mathrm{arg min}_{x \in \chi_{n}(t)}\norm{x - f(t)}_H.
  \end{equation*}
  To verify that $\pi_{n}$ is measurable, we will first show that $\pi_{n}$ is continuous when restricted to large sets, hence rendering it a measurable function on such sets. Denoting the Lebesgue measure by $\lambda$, we proceed as follows:
  \begin{enumerate}[leftmargin = 4ex, label = \arabic*.]
    \item For some $0<\delta < 1$ let $K_{1}$ be the compact set from Lusin's \cref{th:Lusin} that renders $f$ continuous on $K_{1}$ and satisfies $\lambda\big([0,T]\!\setminus \!K_{1}\big)< \delta$.
    \item We can now iteratively start with $K_{k}$ and find a compact set $K_{k+1}$ s.t.\ $f$ is continuous on $K_{k+1}$ and $\lambda\bigl([0,T]\!\setminus \!\bigcup_{1\leq l \leq k+1}K_{l}\bigr) < \delta^{k+1}$.
    \item We define $K\coloneq \bigcup_{k \in \N}K_{k}$, which is a measurable set. By construction, $\lambda\bigl([0,T]\!\setminus \!K\bigr)=0$ and $f$ is continuous on any restriction to some $K_{k}$. We note, that $K$ is dense in $[0,T]$, since otherwise $[0,T]\!\setminus \!K$ would need to contain an interval, hence having positive measure.
    \item $\chi_n|_{K_k}$ is continuous for $n$ large enough:

          By construction, for all $t\in K_k$,
          \begin{equation*}
            B\big(f(t),\nicefrac{\varepsilon}{3}\big)\cap F(u(t))\neq \emptyset.
          \end{equation*}
          Since $f|_{K_k}$, and $F \circ u$ are continuous, by \cref{lem:hdcont}, it follows that
          \begin{equation*}
            B[f(s),\nicefrac{\varepsilon}{3}]\cap F(u(s))\to B[f(t),\nicefrac{\varepsilon}{3}]\cap F(u(t)) \quad \text{as}\ s\to t.
          \end{equation*}
          Thus, there exists $\eta>0$ s.t.\ for $s\in K_k$ with $|s-t|\leq \eta$ we have
          \begin{equation*}
            B[f(s),\nicefrac{\varepsilon}{3}]\cap F(u(s))\neq \emptyset.
          \end{equation*}
          Similarly, since $F\circ u_n \to F\circ u$ uniformly, we find $n_0(t)\in \N$ s.t.
          \begin{equation*}
            B[f(s),\nicefrac{2\varepsilon}{3}]\cap F(u_n(s))\neq \emptyset
          \end{equation*}
          for all $s\in K_k$ with  $|s-t|\leq \eta$ for all $n\geq n_0$. Since $\bigl((t-\nicefrac{\eta}{2},t+\nicefrac{\eta}{2})\bigr)_{t\in K_k}$ constitutes an open cover for the compact $K_k$, we find $n_0\in \N$ s.t.\ for all $n\geq n_0$ and $s\in K_k$, $B[f(s),\nicefrac{2\varepsilon}{3}]\cap F(u_n(s))\neq \emptyset$. In particular, $\chi_n(t)\neq \emptyset$ for all $n\geq n_0$ and, by \cref{lem:hdcont},
    \[
   K_k \ni t \mapsto  B[f(t),\varepsilon]\cap F(u_n(t))=\chi_n(t)
    \]
    is continuous in the Hausdorff-distance.
    \item Next, we argue that $\pi_n|_{K_k}$ is continuous. For this, we recall from \cref{rem:metrproj} that
          \begin{alignat*}{5}
            p_C &\colon H &\to C, &\quad x &&\mapsto \mathrm{arg min}_{c \in C}\norm{c - x}_{H}
            \intertext{is Lipschitz-continuous for any fixed, nonempty, closed and convex subset $C\subseteq H$. We define}
            p_{k,C} &\colon K_{k} &\to C, &\quad t &&\mapsto p_{C}(f(t)).
          \end{alignat*}
          Note that $f[K_k]$, as the continuous image of a compact set, is bounded. Then we find $R>0$ such that
          \[
            \chi_n[K_k]\subseteq \bigcup_{y\in f[K_k]} B(y,2\varepsilon)\subseteq B[0,R].
          \]
Finally, let $t,s\in K_k$, then appealing to \cref{it:ProjEst} of \cref{rem:metrproj}, we estimate
          \begin{align*}
             \qquad\norm{\pi_n(t) \!- \!\pi_n(s)} & = \|p_{k, \chi_n(t)}(t) - p_{k,\chi_n(s)}(s)\| \\
             & \leq \|p_{k, \chi_n(t)}(t) - p_{k,\chi_n(t)}(s)\| + \|p_{k, \chi_n(t)}(s) - p_{k,\chi_n(s)}(s)\| \\
             & = \| p_{\chi_n(t)}(f(t))- p_{\chi_n(t)}(f(s))\| + \|p_{\chi_n(t)}(f(s))-p_{\chi_n(s)}(f(s))\| \\
             & \leq \|f(t)-f(s)\| + \sqrt{(4 \|f(s)\|+2R)\Hd(\chi_n(t),\chi_n(s))}.
          \end{align*}
By our previous observation, the latter tends to $0$, as $t\to s$.
          \item Iterating this process, we retain $\pi_{n}\colon [0,T]\to H$ as a pointwise limit of measurable functions, rendering it a measurable function itself (see~\cite[prop.~3.1.3]{Waurick2022}).
  \end{enumerate}
  It remains to show the desired inequality: Indeed, since $\pi_{n}$ satisfies the pointwise estimate $\norm{\pi_{n}(t)- f(t)}_{H}\leq \epsilon$ for almost all $t \in [0,T]$, for the $\Leb^{2}$-norm we obtain: $\norm{\pi_{n}-f}_{\Leb^{2}(0,T;H)}\leq \epsilon \sqrt{T}$. In particular, $\pi_{n}\in \Leb^{2}(0,T;H)$.
\end{proof}

\begin{remark}\label{rem:actualthm3.3.1} Note that a close inspection of the proof reveals that the theorem statement can be replaced by the following: For all $\varepsilon>0$ there is $\delta>0$ s.t.\ for all $u,v\in \mathcal{C}([0,T];H)$ and  $f\in \Sel F(u)$ with
\[
    \sup_{t\in [0,T]} \Hd\bigl( F (u(t)), F(v(t))\bigr)\leq \delta,
\]
we find $g\in \Sel F(v)$ s.t.\
\[
   \| f-g\|\leq \varepsilon.
\]
\end{remark}
%%%%%%%%%%%%%%%%%%%%%%%%%%%%%%
%    S E C T I O N
\section{Uniform convergence of solutions}
\label{sec:UniformConvergence}

To establish local well-posedness of the coupled problem \eqref{eq:IVP}, it will become necessary to establish uniform convergence of solutions $u_{n}$ and $v_{n}$ of the decoupled problems
\[
 \begin{cases}   \tfrac{\dd}{\dx[t]}u + E u \hspace{-1.5ex}&= f_n\\
    \hfill u(0) \hspace{-1.5ex}&= u_{0}
 \end{cases}
 \quad \text{ and }\quad
 \begin{cases}     \tfrac{\dd}{\dx[t]}v + A(\argdot) v \hspace{-1.5ex}&= g_n\\
    \hfill v(0) \hspace{-1.5ex}&= v_{0}
\end{cases}
\]
associated to some suitable right-hand sides $f_{n}$ and $g_{n}$ respectively. Accordingly, we split this section into two parts.

\subsection{Uniform convergence of semigroup solutions}
\label{subsec:UniformConvergenceOfSG}
For technical reasons, we will require uniform convergence of solutions $u_{n}$ to the problem
\begin{equation}
  \label{eq:CauchyProblem}
  \left\{
  \begin{aligned}
    \tfrac{\dd}{\dx[t]}u_{n} + E u_{n} &= f_{n} &&\text{on}\ (0,T_{0}) \text{,}\\
    u_{n}(0) &= u_{0} &&\text{in}\ H\text{.}
  \end{aligned}
  \right.
\end{equation}
This problem was touched upon in \cite[sec.~3]{Simsen2024} for the special case where $-E$ is the generator of the Schr\"odinger semigroup. There it appears to have been wrongly claimed that the solutions of problem \eqref{eq:CauchyProblem} adhere to the concept of maximal regularity. This section elaborates on the subject and showcases an example, where Lipschitz-continuity of solutions fails and hence, a straigtforward application of the Arzel\`{a}--Ascoli-theorem to obtain a continuous limit solution fails. We also present a result that salvages the convergence under (necessarily) stricter assumptions.

\begin{example}
  \label{ex:Counterexample}
  Let $\Delta$ be the Dirichlet--Laplacian on $H = \Leb^{2}(0,\pi)$ with $\dom(\Delta)=\Sob^2(0,\pi)\cap \Sob_0^1(0,\pi)$. We denote by $S$ the semigroup generated by $\iu\Delta$. We construct $f\in \Leb^{2}(0,\pi)$
  so that $t\mapsto S(t)f$ is not Lipschitz-continuous.  It is well-known that the normalized eigenfunctions and eigenvalues of $\Delta$ are
  \begin{equation*}
    \phi_n(x)=\sqrt{\tfrac{2}{\pi}}\sin(nx),\qquad \lambda_n=-n^2,\qquad n\in\mathbb{N}.
  \end{equation*}
Hence, $S(t)\phi_n=e^{-\iu tn^2}\phi_n$. For $n\in \N$, define $a_n \coloneqq (1+n^2)^{-3/4}$ and let $f\coloneq \sum_{n=1}^\infty a_n\,\phi_n \in L^2(0,\pi)$. Since
  \begin{equation*}
    \sum_{n=1}^\infty \lambda_n^2 a_n^2
    =\sum_{n=1}^\infty n^4(1+n^2)^{-3/2}
    \sim \sum_{n=1}^\infty n = +\infty,
  \end{equation*}
  $f\notin \operatorname{dom}(-\Delta)$. 
  
  Let $u\coloneqq S(\cdot)f$.
  By orthogonality of the eigenbasis,
  \begin{equation*}
\norm{u(t)-u(0)}_{H}^2   = \norm{u(t)-f}_{H}^2=\sum_{n=1}^\infty |e^{-itn^2}-1|^2\,a_n^2 .
  \end{equation*}
  We find $c>0$ such that, for $\abs{\theta} \le \tfrac{1}{2}$,   \begin{equation*}
    \abs{e^{\iu\theta}-1} \ge c \abs{\theta}.
  \end{equation*}
Hence, for $\abs{t}>0$ and integers $1\le n\le N(t)\coloneqq \bigl\lfloor (2\abs{t})^{-1/2}\bigr\rfloor$, we have $\abs{t}n^2\le\tfrac{1}{2}$ and
  \begin{equation*}
    \abs{\e^{-\iu tn^2}-1}\ge c \abs{t} n^2 .
  \end{equation*}
Thus with $N(t)\simeq (2|t|)^{-1/2}$,
  \begin{align*}
   \norm{u(t)-f}_{H}^2
   & \ge c^2 \abs{t}^2 \sum_{n=1}^{N(t)} n^4 a_n^2
    = c^2 \abs{t}^2 \sum_{n=1}^{N(t)} n^4 (1+n^2)^{-3/2}\\ & \gtrsim   \abs{t}^2\sum_{n=1}^{N(t)} n \sim   \abs{t}^2 N(t)^2 \simeq \abs{t}.
  \end{align*}
Consequently,
  \begin{equation*}
    \frac{\norm{u(t)-u(0)}_{\Leb^2}}{\abs{t}} \gtrsim \abs{t}^{-1/2}\xrightarrow[t\to0]{}\infty\text{.}
  \end{equation*}
  In particular, $u$ is not Lipschitz-continuous.
\end{example}

In the following, when not expressed differently explicitly, we endow $\dom(E)$ for a closed linear operator $E$ with its graph scalar product rendering it a Hilbert space.

\begin{lemma}
  \label{th:Equicontinuity}
  Let $-E$ be the generator of a $\mathrm{C}_{0}$-semigroup $(T(t))_{t\geq 0}$ on a Hilbert space $H$ and let $(f_n)_{n\in\mathbb{N}}$ in $\dom(E)$ be bounded.\\
  Then $\big(u_n(\cdot)\big)_n \coloneqq \big(T(\cdot)f_n\big)_n$ is equicontinuous on compact time-intervals. More specifically, on compact time-intervals, all $u_n$ are Lipschitz-continuous sharing a common Lipschitz-constant.
\end{lemma}

\begin{proof}
  Let $n\in \N$ and $T_0>0$. For $f_n\in \operatorname{dom}(E)$, $t\mapsto u_n(t) = T(t)f_n \in \mathcal{C}^1$ satisfies
  \begin{equation*}
    u_n'(t) = \tfrac{\dd}{\dx[t]}T(t)f_n = -T(t)(E f_n).
  \end{equation*}
  Hence, for any $t,s \geq 0$, 
  \begin{equation*}
    u_n(t)-u_n(s) = \medint\int_s^t u_n'(\tau) \dx[\tau] =- \medint\int_s^t T(\tau)(E f_n) \dx[\tau] .
  \end{equation*}
As $\big(T(t)\big)_{t \in [0,T_{0}]}$ is bounded, we infer
  \begin{equation*}
    \norm{u_n(t)-u_n(s)}_{H}
    \le \medint\int_{s}^{t} \norm{T(\tau)(E f_n)}_{H} \dx[\tau]
    \leq C \medint\int_{s}^{t} \norm{E f_n}_{H}\dx[\tau] .
  \end{equation*}
  From the uniform graph-norm bound $\norm{E f_n}_{H}\le M$ for some $M\geq 0$ and all $n\in \N$ we see that
  \begin{equation*}
    \norm{u_n(t)-u_n(s)}_{H} \le CM \abs{t-s}.\qedhere
  \end{equation*}
\end{proof}

\begin{lemma}
  \label{th:ForcingTerm}
  Let $T_{0}>0$. Let $-E$ be the generator of a $\mathrm{C}_{0}$-semigroup $(T(t))_{t\geq 0}$ on the Hilbert space $H$. Let $(F_n)_{n\in\mathbb N}$ in $\Leb^{2}\big(0,T_{0};\operatorname{dom}(E)\big)$ be bounded.
  Let
  \begin{equation*}
    u_n(t) \coloneq \medint\int_0^t T(t-s)F_n(s) \dx[s],\qquad t\in[0,T_{0}].
  \end{equation*}
  Then $(u_n)_{n\in\mathbb N}$ is bounded in $\mathcal{C}^{\nicefrac{1}{2}}\big([0,T_{0}];H\big)\cap \mathcal{C}\big([0,T_{0}];\operatorname{dom}(E)\big)$. In particular, $(u_n)_{n\in\mathbb N}$ is equicontinuous.
\end{lemma}

\begin{proof}
  Let $n\in \N$, $0\le s\le t\le T_{0}$. Using Duhamel's formula we write
  \begin{align*}
    u_n(t)-u_n(s)
    &= \medint\int_0^t T(t-\tau)F_n(\tau)\dx[\tau] - \medint\int_0^s T(s-\tau)F_n(\tau)\dx[\tau]\\
    &= \underbrace{\medint\int_s^t T(t-\tau)F_n(\tau)\dx[\tau]}_{\eqcolon I_{1,n}(s,t)}
      + \underbrace{\medint\int_0^s\big[T(t-\tau)-T(s-\tau)\big]F_n(\tau)\dx[\tau]}_{\eqcolon I_{2,n}(s,t)}\text{.}
  \end{align*}
  We estimate the first term using the Cauchy--Schwarz-inequality and the boundedness of the semigroup (on the compact interval $[0,T_{0}]$):
  \begin{equation*}
    \norm{I_{1,n}(s,t)}_{H}
    \le C \medint\int_s^t \norm{F_n(\tau)}_{H}\dx[\tau]
    \le C \sqrt{t-s} \norm{F_n}_{\Leb^2(0,T_{0};H)} \le C M\sqrt{t-s}.
  \end{equation*}
  For the second term, first observe that for $g \in \operatorname{dom}(E)$, the orbits $T(\argdot)g$ are $\mathcal{C}^{1}$ with values in $H$ and for $h\geq 0$,
  \begin{equation*}
    T(h)g-g = \medint\int_0^h \tfrac{\dd}{\dx[\theta]}T(\theta)g \dx[\theta] =- \medint\int_0^h T(\theta)(Eg)\dx[\theta]\text{.}
  \end{equation*}
  Consequently,
  \begin{equation*}
    \norm{T(h)g - g}_{H} \le C \abs{h} \norm{E g}_{H}.
  \end{equation*}
  By applying this observation with $h=t-s$ and $g=T(s-\tau)F_n(\tau)$ and observing that $T(s-\tau)F_n(\tau)\in \operatorname{dom}(E)$ one infers that for a.e.\ $\tau\in[0,s]$,
  \begin{align*}
    \norm[\big]{\bigl[T(t-\tau) - T(s-\tau)\bigr] F_n(\tau)}_{H}
    &= \norm{T(t-s)g - g}_{H}\\
    &\le C \abs{t-s} \norm{E T(s-\tau) F_n(\tau)}_{H}\\
    &= C \abs{t-s} \norm{T(s-\tau)E F_n(\tau)}_{H}\\
    &\leq C^{2} \abs{t-s} \norm{E F_n(\tau)}_{H}\text{.}
  \end{align*}
  Integrating over $\tau\in[0,s]$ and using the Cauchy--Schwarz-inequality produces an estimate for $I_{2,n}$:
  \begin{equation*}
    \norm{I_{2,n}(s,t)}_{H}
    \le C^{2} \abs{t-s} \medint\int_0^s \norm{E F_n(\tau)}_{H}\dx[\tau]
    \le C^{2} \abs{t-s} \sqrt{s} \norm{E F_n}_{\Leb^2(0,T_{0};H)} .
  \end{equation*}
  Since $s\le T_{0}$ and $\norm{E F_n}_{\Leb^2(0,T_{0};H)}\le M$, we obtain
  \begin{equation*}
    \norm{I_{2,n}(s,t)}_{H} \le C^{2} M \sqrt{T_{0}} \abs{t-s}.
  \end{equation*}
  Combining the two estimates yields
  \begin{equation*}
    \norm{u_n(t)-u_n(s)}_{H} \lesssim \sqrt{t-s} + \sqrt{T_{0}} \abs{t-s},
  \end{equation*}
  uniformly in $n\in \N$. In particular, as $t\downarrow s$, the right-hand side tends to $0$, uniformly in $n\in \N$, proving equicontinuity.
\end{proof}

If we additionally assume compactness of the resolvent, we obtain a suitable compactness result.
\begin{theorem}
  \label{th:CompactnessInC}
  Let $T_{0}>0$ and $-E$ be generator of a $\mathrm{C}_{0}$-semigroup $(T(t))_{t\geq 0}$ with compact resolvent. Let $(F_n)_{n\in \N}$ in $ \Leb^{2}\big(0,T_{0};\operatorname{dom}(E)\big)$  and   
  $(f_n)_{n\in\N}$ in $\operatorname{dom}(E)$ both be bounded.\\
  Then the sequence of (mild) solutions $u_n(t)\coloneq T(t)f_n + \medint\int_0^t T(t-s)F_n(s)\dx[s]$, $t\in[0,T_{0}]$, $n\in \N$, of
  \begin{equation*}
    \left\{
    \begin{aligned}
      \tfrac{\dd}{\dx[t]}u + Eu &= F_{n},\\
      u(0)&= f_{n},
    \end{aligned}
    \right.
  \end{equation*}
  admits a convergent subsequence in $\mathcal{C}\bigl([0,T_{0}];H\bigr)$.\\
  In particular, if both $(F_n)_{n\in \N}$ in $ \Leb^{2}\big(0,T_{0};\operatorname{dom}(E)\big)$ and $(f_n)_{n\in\N}$ in $\operatorname{dom}(E)$ are weakly convergent to some $F$ and $f$ respectively, then $u_n\to u$ in $\mathcal{C}\bigl([0,T_{0}];H\bigr)$, where $u$ satisfies
  \begin{equation*}
    \left\{
    \begin{aligned}
      \tfrac{\dd}{\dx[t]}u + Eu &= F,\\
      u(0)&= f.
    \end{aligned}
    \right.
  \end{equation*}
\end{theorem}

\begin{proof}
  From \cref{th:ForcingTerm} and \cref{th:Equicontinuity} we infer that $(u_n)_n$ is equicontinuous in $\mathcal{C}\bigl([0,T_{0}];H\bigr)$. Moreover, since $(u_n)_n$ is bounded in $C\bigl([0,T_{0}];\operatorname{dom}(E)\bigr)$, for all $t\in [0,T_{0}]$, $(u_n(t))_n$ is bounded in $\operatorname{dom}(E) \hookrightarrow H$. In particular, because $E$ has compact resolvent, that embedding is compact, so for all $t \in [0,T_{0}]$, the set $\{u_n(t) \colon n\in \mathbb{N}\}$ is relatively compact in $H$. Thus, the Arzel\`a--Ascoli theorem (see~\cite[thm.~7.5.7]{Dieudonne1969}) yields that $(u_n)_n$ admits a convergent subsequence.
  
  The last part follows by weakly testing the solution formula for $u_n$. Then it follows that $u_n\to u$ pointwise weakly in $H$. The compactness statement of the theorem establishes the assertion upon using a subsubsequence argument.
\end{proof}

\begin{remark}\label{rem:complcont}
  Subsequently, we will apply \cref{th:CompactnessInC} (and its special case for weakly convergent sequences) for a sequence  $(F_{n})_{n}$ bounded in $\Leb^{2}(0,T_0;H)$ and a constant sequence of initial values $f_{n}= u_{0} \in H$. The considered generator $-E_{-1}$ is then the one of the extrapolated semigroup $(T_{-1}(t))_{t\geq 0}$. Note that compactness of $\dom(E)\hookrightarrow H$ implies the same for $\dom(E_{-1})=H\hookrightarrow \Sob^{-1}(E)$. As a consequence, the corresponding mild solutions (in $H$) and, at the same time, classical solutions (in $\Sob^{-1}(E)$) are relatively compact in $\mathcal{C}\big([0,T_0];\Sob^{-1}(E)\big)$. In particular, if $(F_{n})_{n}$ is weakly convergent to some $F \in \Leb^{2}(0,T_0;H)$, the corresponding limit solution belongs to $\mathcal{C}\big([0,T_0];H\big)$ by Duhamel's formula, yet it is approached by the pre-asymptotic solutions in $\mathcal{C}\big([0,T_0];\Sob^{-1}(E)\big)$.
  \end{remark}

\subsection{Uniform convergence of solutions to the monotone evolution problem}
\label{subsec:UniformConvergenceMP}

In the preceeding \cref{subsec:UniformConvergenceOfSG} we have already established a compactness result for the semigroup part of problem \eqref{eq:IVP}. For solutions to the monotone evolution problem
\begin{equation}
  \left\{
  \label{eq:MonotoneIVP}
  \begin{aligned}
    \tfrac{\dd}{\dx[t]}v + A(t)v &= g,\\
    v(0) &= v_{0},
  \end{aligned}
  \right.
\end{equation}
we quote a result establishing the compactness of the solution family. In order to do so, let $T>0$, $\mathcal{H}$ be a separable real Hilbert space and let $A(\argdot)$ be a maximal monotone operator. Problem \eqref{eq:MonotoneIVP} admits a solution $v_g\in \mathcal{C}\bigl([0,T];\mathcal{H}\bigr)$ for all $g \in \Leb^{2}\bigl(0,T;\mathcal{H}\bigr)$, and $v_{0} \in \mathcal{H}$. 

We want to make use of regularity properties of the solution operator

\begin{equation*}
  \Sigma \colon \Leb^{2}\bigl(0,T;\mathcal{H}\bigr) \to \mathcal{C}\bigl([0,T];\mathcal{H}\bigr),\quad g \mapsto v_{g},
\end{equation*}
where $v_{g}$ is the unique solution of problem \eqref{eq:MonotoneIVP} for given $g \in \Leb^{2}(0,T;\mathcal{H})$.

The compactness result we are after, states that weakly converging right-hand sides are turned into uniformly convergent solutions of problem \eqref{eq:MonotoneIVP}. We simply quote the following statement, which in itself is a corollary of \cite[thm.~4.1]{Guillaume2005}.
\begin{proposition}[{\cite[lem.~1]{SimsenKloeden2017}}]
  \label{th:UniformConvergence}
  Let $\mathcal{H}$ be a separable, real Hilbert space and $A(t)=\partial\phi^{t}$, $t>0$, where 
  $(\phi^{t})_{t>0}$ satisfies \cref{ass:A}. Let $v_{0}\in \mathcal{H}$.\\
  Then $\Sigma$ is {\em completely continuous}, that is, $\Sigma$ maps weakly convergent sequences (in $\Leb^{2}(0,T;\mathcal{H})$) to strongly convergent sequences (in $\mathcal{C}\bigl([0,T];\mathcal{H}\bigr)$).
\end{proposition}

In the literature, various criteria are available to obtain a similar statement as in \cref{th:UniformConvergence}; for instance if the potentials $\phi^{t}$ are of {\em compact type}, see \cite[prop.~3.3]{HuPapa1998}, the inhomogeneity exhibits Lipschitz-properties, see \cite[thm.~6]{PapaSquared1997} or if $(A(s))_{t>0}$ generates a compact evolution process, see \cite[thm.~21]{SimsenWittbold2019}.

%%%%%%%%%%%%%%%%%%%%%%%%%%%%%%
%    S E C T I O N
\section{Existence results}

\subsection{Auxiliary results from functional analysis}\label{sec:auxfa}

Before we turn our attention to the actual existence proof, we gather some results of general nature. In order to show the existence of local solutions to problem \eqref{eq:IVP}, we aim to apply the following fixed point theorem.
\begin{theorem}[{\cite[thm.~3.5]{Vrabie1994}}]
  \label{th:FixedPointTheorem}
  Let $X$ be a Banach space and let $\emptyset \neq K \subseteq X$ be weakly compact. Let  $F\colon K \to \mathcal{P}(X)$ be bcc. If the graph of $F$ is weakly $\times$ weakly sequentially closed in $K \times X$, then $F$ admits at least one fixed point, i.e., there exists at least one $u \in K$ such that $u \in F(u)$.
\end{theorem}
Next, we recall an elementary inequality.
\begin{lemma}\label{lem:elem} Let $T_0>0$, $u \in \mathcal{C}[0,T_0]$, $u\geq 0$, $h\in \Leb^1(0,T_0)$, $c\geq 0$. If, for all $0\leq t\leq T_0$,
\[
   u(t)^2 \leq c^2 + \medint\int_{0}^t h(s) u(s) \dx[s],
\]then
\[
    u(t) \leq c + \tfrac{1}{2}\medint\int_0^t h(s)\dx[s].
\]
\end{lemma}
\begin{proof}
For $t\in [0,T_0]$, we introduce
\[
  X(t) \coloneqq \medint\int_{0}^t h(s) u(s) \dx[s].
\]
The corresponding weak derivative is
\[
   X'(s) = h(s) u(s) \leq h(s) \sqrt{ c^2 + X(s)}.
\]
Hence, integration of $X'(s)\big(c^2 + X(s)\big)^{-\frac{1}{2}} \leq  h(s) $ over  $s\in  [0,t]$ yields
\[
   2 \sqrt{c^2 +X(t)} - 2 c \leq \medint\int_0^t h(s) \dx[s].
\]
Finally, for $t\in [0,T_0]$ we obtain
\[
   u(t) \leq \sqrt{c^2 +X(t)} \leq c +\tfrac{1}{2} \medint\int_0^t h(s) \dx[s]. \qedhere
\]
\end{proof}

\begin{lemma}\label{lem:linapp} Let $H,K$ be Hilbert spaces, $(T_n)_n, T$ in $L(H,K)$. Assume that $T_n^{\ast} \to T^*$ in the strong operator topology. Then the family $(T_n)_n$ considered as mappings from $\bigl(B_H[0,1],\sigma(H,H')\bigr)$ to $\bigl(K,\sigma(K,K')\bigr)$ is equicontinuous.
\end{lemma}
\begin{proof}
  Let $\varepsilon>0$, $G\subseteq K$ finite and define
  \[
     V\coloneqq \{ z\in K; \forall y \in G \colon |\langle z,y\rangle|<\varepsilon\}.
  \]
  By assumption, $(T_n^*y)_n$ is convergent for every $y \in G$, so that 
  \[
     M\coloneqq \overline{ \bigcup \{ T_n^*y; n\in \N, y\in G\}}\subseteq H
  \]
  is norm-compact. In particular, $M$ is totally bounded and we find $F\subseteq H$ finite s.t.
  \[
      M \subseteq \bigcup_{\phi\in F} B(\phi,\nicefrac{\varepsilon}{2}).
  \]
  For $y \in G$ and $n\in \N$ we find $\phi \in F$ such that $\|T_n^*y- \phi\|<\nicefrac{\varepsilon}{2}$. Now define the weakly open neighbourhood
  \begin{equation*}
    U\coloneqq \{ x\in H; |\langle x,\phi\rangle|<\nicefrac{\varepsilon}{2} \text{ for all }\phi \in F\}\cap B_{H}[0,1]
  \end{equation*}
  and compute for $x \in U$
  \[
     | \langle T_n x, y\rangle| \leq      | \langle  x, T_n^* y-\phi \rangle| + |\langle x,\phi\rangle|\leq \|T_n^*y-\phi\|+|\langle x,\phi\rangle|< \tfrac{\varepsilon}{2}+\tfrac{\varepsilon}{2} = \varepsilon.
  \]
  Thus $T_{n}[U]\subseteq V$ for all $n\in\N$, that is, $(T_n)_n$ is equicontinuous at $0$, which by linearity is enough to establish equicontinuity everywhere.
\end{proof}
\begin{remark}\label{rem:Tlambda}\phantom{.}
  \begin{enumerate}[label = (\roman*), leftmargin = 5ex]
    \item The result in \cref{lem:linapp} also applies to nets $(T_\lambda)_{\lambda\geq \lambda_0}$ for $\lambda\to \infty$ as long as $[\lambda_0,\infty) \ni \lambda \mapsto T_\lambda^* \in \mathcal{L}(K,H)$ is norm-continuous. Indeed, in that case, $M$ in the proof can be chosen to be the union of $M_y\coloneqq \overline{\{ T_\lambda^* y; \lambda \geq \lambda_0\}}$ over $y\in G$. By continuity of $\lambda\mapsto T_{\lambda}^*$, on bounded, closed intervals, the considered orbits are compact; the limiting behaviour ensures compactness of $M_y$ and, hence, of $M$. The remaining steps of the proof are similar.
    \item It is easy to see that the statement in \cref{lem:linapp} remains true, if $B_H[0,1]$ is replaced by $B_H[0,R]$ for any $R>0$.
  \end{enumerate}
\end{remark}

\begin{lemma}\label{lem:weakstrong} Let $X$ be a metric space and $H,\mathcal{H}$ Hilbert spaces with $H$ being separable. Let $F\colon H\times \mathcal{H}\to X$ be continuous and assume
  \begin{itemize}[leftmargin = 4ex]
    \item for all $R>0$ and all $y \in B_{\mathcal{H}}[0,R]$ the function
          \[
            F_y \colon \bigl(B_{H}[0,R],\sigma(H,H')\bigr)\to \mathcal{P}(X), \quad x\mapsto F(x,y)
          \]
          is continuous.
    \item $[0,\infty)\ni \lambda\mapsto S_\lambda \in \mathcal{L}(H)$ is continuous.
    \item $S_\lambda^* \to\colon S^* \in \mathcal{L}(H)$ in the strong operator topology as $\lambda\to \infty$.
  \end{itemize}
  Then, for all $R>0$ and $\mathcal{K} \subseteq \mathcal{H}$ compact, we obtain for $S\coloneq S^{\ast\ast}$
\[
 \forall\varepsilon\!>\!0 \,\exists \lambda_0 \!> \!0 \,\forall \lambda \!\geq \!\lambda_0, x \!\in \!B_{H}[0,R], y \!\in  \!\mathcal{K} \colon \Hd\bigl(F(S_\lambda x, y),F(Sx,y)\bigr)\leq \varepsilon.
\]
\end{lemma}
\begin{proof} The uniform boundedness principle ensures that $S_\cdot$ is bounded. Hence, $\bigcup_{\lambda\geq 0} S_\lambda [B_{H}[0,R]]\subseteq H$ is bounded by $R'>0$, say. By \cref{lem:linapp} with $\sigma=\sigma(H,H')$,
  \begin{equation*}
    \big(B_{H}[0,R],\sigma\big)\times \mathcal{K} \ni (x,y)\mapsto (S_\lambda x,y) \in \big(B_{H}[0,R'],\sigma\big)\times \mathcal{K}
  \end{equation*}
  is equicontinuous.
Let $\varepsilon>0$. By Banach--Alaoglu, the reflexivity of $H$ and its separability, it follows that $\big(B_{H}[0,R'],\sigma\big)$ is metrisable. Now, $F(x,\cdot)$ is continuous and $F_y=F(\cdot,y)$ is continuous (under the weak topology). Hence,
\[
F\colon  \big(B_{H}[0,R'],\sigma\big) \times \mathcal{K} \to X
\]
is continuous. Thus, by \cref{lem:easycons}, as $\big(B_{H}[0,R'],\sigma\big)\times \mathcal{K}$ is compact, $F(S_\lambda x,y)\to F(S x,y)$ uniformly in $(x,y)$ as $\lambda\to\infty$, which implies the assertion.
\end{proof}
\begin{remark}
  A particular instance of the application of the latter lemma is the following: Let $[0,\infty)\ni \lambda\mapsto S_\lambda \in \mathcal{L}(H)$ be continuous and such that $S_\lambda^* \to\colon S^* \in L(H)$ in the strong operator topology. If $(u_n)_n$ in $\mathcal{C}\big([0,T];H\big)$ is bounded  and $(v_n)_n$ in $\mathcal{C}\big([0,T];\mathcal{H}\big)$ is convergent, then
  \begin{multline}\label{eq:need}
  \forall\varepsilon \!> \!0 \, \exists \lambda_0 \!> \!0 \,\forall \lambda \!\geq \!\lambda_0, n \!\in \!\N, t \!\in \![0,T]\colon \\  \Hd\big(F(u_0+S_\lambda u_{n}(t),v_n(t)),F(u_0+u_n(t),v_n(t))\big)\leq \varepsilon.
\end{multline}
Indeed, this is implied by the fact that $\overline{\bigcup_{n\in \N} u_n[0,T]}$ is bounded and, due to convergence of $(v_n)_n$, compactness of $\overline{\bigcup_{n\in \N} v_n[0,T]}$.
\end{remark}

\subsection{Local existence}
\label{sec:LocalExist}

The idea for the proof of our local existence result is motivated from the autonomous case for the generalized Schr\"odinger--Debye-system, studied in \cite[thm.~37]{Simsen2024}. However, several results in \cite{Simsen2024} appear to require a modification, pertaining to the lack of maximal regularity of solutions of the Schr\"odinger equation, mandating several complications in the proof. Additionally, problem \eqref{eq:IVP} is now in comparison non-autonomous, asking for new ingredients. The time-dependence of $A$ can be dealt with by carefully handling selections, for the semigroup part --- due to lack of parabolic regularity in applications such as the Schr\"odinger equation --- stronger properties on the forcing terms have to be demanded. These can be checked in a relevant example class, see~\cref{subsec:RHS}.

\begin{definition}\label{def:whc}
 Let $H,\mathcal{H},X$ be Hilbert spaces, $F\colon H\times \mathcal{H}\to \mathcal{P}(X)$ bcc. We call $F$ {\em half weakly continuous (w.r.t.\ the first variable)}, if for all $y\in \mathcal{H}$ and $R>0$, the map
 \[
    F_y \colon \bigl(B_H[0,R],\sigma(H,H')\bigr)\to \mathcal{P}(X), \quad x\mapsto F(x,y)
 \]
 is continuous w.r.t.\ the Hausdorff-distance in $\mathcal{P}(X)$.
\end{definition}

The main result of this section is the following:
\begin{theorem}[local existence]
  \label{th:LocalExist}
  Let $H$ and $\mathcal{H}$ be separable real Hilbert spaces. Let $-E$ be generator of a $\mathrm{C}_{0}$-semigroup $(T(t))_{t\geq 0}$ with compact resolvent and let $A=(A(t))_{t>0}$ be a family of univalued operators $A(t)=\partial \phi^t$ with $\phi^t$ satisfying \cref{ass:A}, $T_m>0$. Let $F\colon H \times \mathcal{H} \to \mathcal{P}(H)$ and $G\colon H \times \mathcal{H} \to \mathcal{P}(\mathcal{H})$
  \begin{itemize}[leftmargin = 4ex]
    \item be bcc functions.
    \item be Hausdorff-continuous.
    \item be half weakly continuous (both w.r.t.\ the first variable).
    \item map bounded sets of $H \times \mathcal{H}$ into bounded sets of $H$ and $\mathcal{H}$ respectively.
  \end{itemize}
  Then for any bounded set  $B \subseteq H \times \mathcal{H}$ there exists $T_{m}\geq T_{0} >0$\footnote{Here and in the following we denote the maximal existence time by $T_{m}$ instead of $T$ to avoid any confusion with the semigroup $T(\argdot)$.} s.t.\ for each  $(u_{0}, v_{0}) \in B$ there exists at least one strong solution  $(u,v)$ of \eqref{eq:IVP} on  $[0, T_{0}]$.
\end{theorem}

Since the proof is technical and long, we segregate individual steps and formulate them as separate lemmata.

\begin{lemma}
  \label{th:LocalExistSingleValued}
  Under the assumptions of \cref{th:LocalExist}, let $B \subseteq H \times \mathcal{H}$ be bounded, $T_0>0$. Then for any $(f,g) \in \Leb^{2}(0,T_{0};{H})\times \Leb^{2}(0,T_{0};\mathcal{H})$ and $(u_0,v_0)\in B$, let $(u_f,u_g) \in  \mathcal{C}\bigl([0,T_{0}];H\bigr)\times \mathcal{C}\bigl([0,T_{0}];\mathcal{H}\bigr)$ solve
  \begin{equation}
    \label{eq:SingleValuedProblem}
    \left\{
    \begin{aligned}
      \tfrac{\dd}{\dx[t]}u + E u &= f &&\text{on}\ (0,T_{0})\times H,\\
      \tfrac{\dd}{\dx[t]}v + A(t)v &= g &&\text{on}\ (0,T_{0})\times \mathcal{H},\\
       u(0) &= u_{0} &&\text{in}\ H,\\
       v(0) &= v_{0} &&\text{in}\ \mathcal{H}.
    \end{aligned}
    \right.
  \end{equation}
  Then
  \begin{align}
    \label{eq:BoundedSolutions}
    \begin{aligned}
      &\max \bigl\{\norm{u_{f}}_{\mathcal{C}([0,T_{0}];H)}, \norm{v_{g}}_{\mathcal{C}([0,T_{0}],\mathcal{H})}\bigr\} \\
      &\qquad\leq m-1 + \sqrt{T_{0}}\max \bigl\{\norm{f}_{\Leb^{2}(0,T_{0};H)}, \norm{g}_{\Leb^{2}(0,T_{0};\mathcal{H})}\bigr\}\text{,}
    \end{aligned}
  \end{align}
  where 
  \begin{equation}
    \label{eq:m}    m \coloneq \widetilde{C} \beta + 1
\text{ with } \widetilde{C}\coloneqq \sup_{s\in [0,T_{m}]}\norm{T(s)},\; \beta \coloneqq \sup_{(a,b)\in B} \big\{\norm{a}_{H},\norm{b}_{\mathcal{H}}\big\}.
   \end{equation}
\end{lemma}

\begin{proof}
  Let $(f,g) \in \Leb^{2}(0,T_{0};{H})\times \Leb^{2}(0,T_{0};\mathcal{H})$, $(u_0,v_0)\in B$ and $T_0>0$. Note that unique solvability is established in \cref{th:SemiGroupSolution} and \cref{th:MonotoneSolution}, respectively.
  We proceed to show the estimates:
  \begin{enumerate}[leftmargin=4ex, label = (\roman*)]
    \item \label{it:SGestimate} For $u_{f}$ we can estimate for $0 \leq t \leq T_{0}$:
          \begin{align*}
            \norm{u_{f}(t)}_{H}
            &\leq \norm{T(t)u_{0}}_{H} + \norm[\Big]{\medint\int_{0}^{t}T(t-s)f(s)\dx[s]}_{H}\\
            &\leq \norm{T(t)}_{\mathcal{L}(H)}\norm{u_{0}}_{H} + \medint\int_{0}^{t} \norm{T(t-s)f(s)}_{H}\dx[s]\\
            &\leq \widetilde{C}\norm{u_{0}}_{H} + \widetilde{C}\medint\int_{0}^{t} \norm{f(s)}_{H}\dx[s]\\
            &\leq \widetilde{C}\norm{u_{0}}_{H} + \widetilde{C}\sqrt{T_{0}}\norm{f}_{\Leb^{2}(0,T_{0};H)}\\
            &\leq m -1 + \sqrt{T_{0}}\norm{f}_{\Leb^{2}(0,T_{0};H)}
          \end{align*}
    \item \label{it:Monotestimate} For $v_{g}$ we first note, that $v_{g}$ is absolutely continuous on $[\delta,T_{0}]$ by \cref{th:MonotoneSolution} and from monotonicity of $A(t)$ for a.e. $\delta \leq t \leq T_{0}$ using $0 \in \dom(A(t))$, we obtain
    \begin{align*}
      \tfrac{1}{2}\tfrac{\dd}{\dx[t]}\dualprod[\big]{v_{g}(t)}{v_{g}(t)}_{\mathcal{H},\mathcal{H}}
            &=\dualprod[\big]{\tfrac{\dd}{\dx[t]}v_{g}(t)}{v_{g}(t)}_{\mathcal{H},\mathcal{H}}\\
            & = - \underbrace{\dualprod[\big]{A(t)v_{g}(t)}{v_{g}(t)}_{\mathcal{H},\mathcal{H}}}_{\geq 0}
            + \dualprod[\big]{v_{g}(t)}{g(t)}_{\mathcal{H},\mathcal{H}}\\
            & \leq \norm{v_{g}(t)}_{\mathcal{H}}\norm{g(t)}_{\mathcal{H}}.
\end{align*}
    Integration over $\delta \leq t \leq T_{0}$ yields
          \begin{align*}
            \norm{v_{g}(t)}_{\mathcal{H}}^{2}
            &\leq \norm{v_{g}(\delta)}_{\mathcal{H}}^{2} +2\medint\int_{\delta}^{t}\norm{v_{g}(s)}_{\mathcal{H}}\norm{g(s)}_{\mathcal{H}}\dx[s]\text{.}
            \intertext{Since, again by \cref{th:MonotoneSolution}, the solution $v_{g}$ is continuous on $[0,T_{0}]$, by letting $\delta \downarrow 0$, we obtain for $0 < t \leq T_{0}$}
              \norm{v_{g}(t)}_{\mathcal{H}}^{2}
            &\leq \norm{v(0)}_{\mathcal{H}}^{2} + 2\medint\int_{0}^{t}\norm{v_{g}(s)}_{\mathcal{H}}\norm{g(s)}_{\mathcal{H}}\dx[s]\text{.}
          \end{align*}
          Hence, by \cref{lem:elem}, we infer
          \begin{align*}
            \norm{v_{g}(t)}_{\mathcal{H}} & \leq \norm{v(0)}_{\mathcal{H}} +  \int_0^t \norm{g(s)}_{\mathcal{H}}\dx[s] \\
            & \leq m-1 + \sqrt{T_0} \|g\|_{\Leb^{2}(0,T_{0}.
            \mathcal H)}.\qedhere
          \end{align*}
  \end{enumerate}
\end{proof}
\Cref{th:LocalExistSingleValued} allows us to define the solution map
  \begin{equation}
    \label{eq:SolutionMap}
    S\colon B\to \mathcal{C}\bigl([0,T_{0}];H\bigr)\times \mathcal{C}\bigl([0,T_{0}];\mathcal{H}\bigr)\text{,}\qquad (f,g)\mapsto (u_{f},v_{g}),
  \end{equation}
corresponding to problem \eqref{eq:SingleValuedProblem}. Ultimately, to prove \cref{th:LocalExist}, we want to make use of the fixed point \cref{th:FixedPointTheorem}. The next lemma will allow us to define a fixed point map, provided that we restrict to a sufficiently small set.
\begin{lemma}
  \label{th:PhiSelfmap}
  Under the assumptions of \cref{th:LocalExist}, there exists $T_0>0$ s.t.
  \begin{equation}
    \label{eq:M}
    M\coloneq {{B}}_{\Leb^{2}(0,T_{0};H)}[0,m]\times {{B}}_{\Leb^{2}(0,T_{0};\mathcal{H})}[0,m] \text{,}
  \end{equation}
  where $m$ is given as in \eqref{eq:m}, satisfies: for $(f,g) \in M$, any pair of measurable selections $\overline{f}\in \Sel F\bigl(S(f,g)\bigr)$ and $\overline{g} \in \Sel G\bigl(S(f,g)\bigr)$ satisfies $(\overline{f},\overline{g}) \in M$.
\end{lemma}
We point out, that the sets of selections in question are non-empty appealing to \cref{th:SelectionBySeparability}, since the Bochner--Lebesgue space $\Leb^{2}(0,T_{m};H)$ over a separable Hilbert space $H$ is separable again.

\begin{proof}
  We first argue that measurable selections $\overline{f}$ and $\overline{g}$ are indeed in $\Leb^{2}$; in fact, they are in $\Leb^{\infty}(0,T_{m};H)$:
  
  Let  $(u,v) = S(f,g)$. Since $F$ is bcc, we know that $F\bigl((u,v)(t)\bigr)$ is non-empty, bounded, closed and convex. Since the pair $(u,v)$ is continuous, the ranges $u\bigl([0,T_{m}]\bigr) \subseteq H$ and $v\bigl([0,T_{m}]\bigr)\subseteq \mathcal{H}$ are compact, hence bounded; consequently $F\circ (u,v) \bigl([0,T_{m}]\bigr)$ and $G \circ (u,v) \bigl([0,T_{m}]\bigr)$ are bounded, because $F$ and $G$ are bounded (i.e., map bounded sets into bounded sets). Hence, any $\overline{f} \in \Sel F(u,v)$ and $\overline{g} \in \Sel G(u,v)$ are measurable and bounded and consequently in $\Leb^{\infty}(0,T_{m};H)$ and $\Leb^{\infty}(0,T_{m};\mathcal{H})$ respectively.

  For any $(\overline{f},\overline{g}) \in M$, the corresponding solutions $u_{\overline{f}}$ and $v_{\overline{g}}$ defined by $(u_{\overline{f}},v_{\overline{g}}) \coloneqq S(\overline{f},\overline{g})$ are uniformly bounded in $\norm{\argdot}_{H}$- and $\norm{\argdot}_{\mathcal{H}}$-norm by \eqref{eq:BoundedSolutions}. Since $F$ and $G$ are bounded, there exists $r>0$ such that
  \begin{equation}
    \label{eq:Balls}
    \bigcup_{(u,v)\in M}F(u,v)\big([0,T_{0}]\big) \subseteq {{B}}_{H}[0,r]\quad \text{and} \quad \bigcup_{(u,v)\in M}G(u,v)\big([0,T_{0}]\big) \subseteq {{B}}_{\mathcal{H}}[0,r]\text{.}
  \end{equation}
  By choosing $T_{0}$ s.t.\ $0 < T_{0} \leq \big(\nicefrac{m}{r}\big)^{2}$ we estimate
  \begin{equation*}
    \norm[\big]{\overline{f}}_{\Leb^{2}(0, T_0,\mathcal{H})} = \Bigl(\medint\int_{0}^{T_0} \norm[\big]{\overline{f}(t)}_{\mathcal{H}}^{2} \dx[t]\Bigr)^{1/2} \leq (r^2 T_0)^{1/2} \leq m.
  \end{equation*}
  Similarly, we obtain $\norm{\overline{g}}_{\Leb^{2}(0, T_0, H)} \leq m$. Consequently, $\bigl(\overline{f}, \overline{g}\bigr) \in M$.
\end{proof}

Ultimately, we want to verify that the following map admits a fixed point.
\begin{equation}
  \label{eq:FixedPointMap}
  \Phi\colon M \to \mathcal{P}(M)\text{,}\quad (f,g)\mapsto \Sel F\bigl(S(f,g)\bigr)\times \Sel G\bigl(S(f,g)\bigr)\text{,}
\end{equation}
where $M$ is defined as in \eqref{eq:M}. \Cref{th:PhiSelfmap} will help to show that $\Phi$ is well-defined. As we aim to apply \cref{th:FixedPointTheorem}, we check that $\Phi$ exhibits several properties. We verify all but one in the next lemma; the remaining one we verify in \cref{th:WeakSeqMeasurability}.

\begin{lemma}
  \label{th:PropertiesOfPhi}
  Under the assumptions of \cref{th:LocalExist}, the map $\Phi$ from \eqref{eq:FixedPointMap} is a well-defined, bcc map.
\end{lemma}

\begin{proof}
  We split the argument into $3$ steps:
  \begin{enumerate}[leftmargin=4ex, label = (\roman*)]
    \item The action of $\Phi$ is never the empty set.
    
          Let $(f,g) \in M$ and $(u,v) = S(f,g)$, where $S$ is the solution map, see~\eqref{eq:SolutionMap}. Since $(u,v) \in \mathcal{C}\bigl([0, T_{m}], H\bigr) \times \mathcal{C}\bigl([0, T_{m}], \mathcal{H}\bigr)$, $u$ and $v$  are measurable functions. Since $F$ and $G$ are upper semicontinuous (because they are Hausdorff-continuous and bcc), $G(u,v) \colon [0, T_0] \to \mathcal{P}(H)$ and $F(u,v) \colon [0, T_0] \to \mathcal{P}(\mathcal{H})$ are measurable. Since $H$ and $\mathcal{H}$ are separable, \cref{th:SelectionBySeparability} shows that $\Sel F(u,v) \neq \emptyset \neq \Sel G(u,v)$.
    \item $\Phi$ maps elements of  $M$ into subsets of  $M$.

          This follows immediately from \cref{th:PhiSelfmap}.
    \item $\Phi$ assumes convex and closed values.
    
           For convexity, let $f_{1}, f_{2} \in \Sel F(u,v)$ and $0 < \alpha < 1 $. Then, $\alpha f_1 + (1 - \alpha) f_2$ is measurable, since $f_{1}$ and $f_{2}$ are measurable.  Moreover,
          \begin{equation*}
            \alpha f_1(t) + (1 - \alpha) f_2(t) \in F\big(u(t), v(t)\big) \quad \text{a.e.~in}\ [0,T_{0}]\text{,}
          \end{equation*}
          since $F$ is bcc; in particular, $F\bigl(u(t), v(t)\bigr)$ is convex. The same procedure shows that $\Sel G(u,v)$ is in fact convex.
          
          To verify closedness in $\Leb^{2}(0, T_{0}; H) \times \Leb^{2}(0, T_{0}; \mathcal{H})$, consider a sequence $( f_{n})_{n \in \mathbb{N}} $ in $ \Sel F(u,v)$ s.t.\ $f_{n} \to \overline{f}$ in $\Leb^{2}(0, T_{0}, H)$. Evidently, $\overline{f} \in \Leb^{2}(0, T_{0},H)$ is measurable. By Fischer--Riesz, there exists a subsequence $(f_{n_k})_{k}$ converging pointwise almost everywhere in $H$. Thus,
          \begin{equation*}
            \overline{f}(t) \in \overline{F\bigl(u(t), v(t)\bigr)} = F\bigl(u(t), v(t)\bigr) \quad \text{a.e.~in}\ [0, T_{0}]\text{,}
          \end{equation*}
          because $F$ is bcc. Hence, $\Sel F(u,v)$ is closed. The same procedure shows that $\Sel G(u,v)$ is closed. Consequently, $\Phi$ assumes closed and convex values.\qedhere
  \end{enumerate}
\end{proof}

\begin{remark}\label{rem:wclosed}
  The sets $\Sel F(u,v)$ and $\Sel G(u,v)$ are even weakly compact in $\Leb^{2}(0,T_{0};H)$ and $\Leb^{2}(0,T_{0},\mathcal{H})$ respectively.  Indeed, since $\Leb^{2}(0,T_{0};H)$ is reflexive, it suffices to confirm boundedness and weak closedness. Since $\Sel F(u,v)$ is bounded, convex, and closed, Mazur's lemma implies it is weakly closed and hence weakly compact. The arguments for $\Sel G(u,v)$ are verbatim.
\end{remark}

We can now prove a first convergence result for solutions of problem \eqref{eq:SingleValuedProblem}.

\begin{lemma}
  \label{th:ConvergenceOfSolutions}
  Let $(f_{n},g_{n}) \in M$, where $M$ is as in \eqref{eq:M} and let $(u_{n},v_{n}) = S(f_{n},g_{n})$, where $S$ is the solution operator from \eqref{eq:SolutionMap}.
  Under the assumptions
  \[
    f_{n}\rightharpoonup\colon f \quad\text{in}\ \Leb^{2}(0,T_{0};H)\,\,\text{ and } 
    g_{n}\rightharpoonup\colon g \quad\text{in}\ \Leb^{2}(0,T_{0};\mathcal{H})\text{,}
\]
  the solutions $u_{n}$ and $v_{n}$ converge in the following sense:
  \begin{alignat*}{4}
    u_{n} &\to u &&\quad\text{in}\ \mathcal{C}\bigl([0,T_{0}];\Sob^{-1}(E)\bigr) &&\quad\text{with}\ u \in \mathcal{C}\bigl([0,T_{0}];H\bigr)\\
    v_{n} &\to v &&\quad\text{in}\ \mathcal{C}\bigl([0,T_{0}];\mathcal{H}\bigr) &&\quad\text{with}\ v \in \mathcal{C}\bigl([0,T_{0}];\mathcal{H}\bigr).
  \end{alignat*}
\end{lemma}

\begin{proof}
  We argue separately for the solutions of the partial problems \eqref{eq:CauchyProblem} and \eqref{eq:MonotoneIVP} that make up problem \eqref{eq:SingleValuedProblem}:
          \begin{enumerate}[label = (\roman*), leftmargin = 5ex]
            \item We begin with the monotone part, i.e.~\eqref{eq:MonotoneIVP}.
            
                  By \cref{th:UniformConvergence}, it follows that the solutions $(v_n)_n$ of
                  \begin{equation*}
                    \left\{
                    \begin{aligned}
                      \tfrac{\dd}{\dx[t]}v_{n}(t)+ A(t) v_n &= g_n(t) & \text{ a.e.~on}\ (0, T_0], \\
                      v_n(0) &= v_{0}, &
                    \end{aligned}
                    \right.
                  \end{equation*}
                   converge in $\mathcal{C}([0, T_0];\mathcal{H})$ to some  $\overline v$ in $C([0, T_0]; \mathcal{H})$ being the strong solution of
                  \begin{equation*}
                    \left\{
                    \begin{aligned}
                      \tfrac{\dd}{\dx[t]} \overline{v}(t)+A(t)\overline{v}(t) &= g(t) & \text{a.e.~on}\ [0,T_0],\\
                      \overline{v}(0)&=v_{0}. &
                    \end{aligned}
                    \right.
                  \end{equation*}
            \item Now for the semigroup part, i.e.~\eqref{eq:CauchyProblem}.
            
            It suffices to use \cref{th:CompactnessInC} in the form mentioned in \cref{rem:complcont}. As $(f_n)_n$ is weakly convergent to some $f\in \Leb^{2}(0,T_{0};H)\subseteq\Leb^{2}\big(0,T_{0};\Sob^{-1}(E)\big)$, by \cref{th:CompactnessInC}, we find $\overline{u} \in \mathcal{C}\bigl([0, T_{0}]; \Sob^{-1}(E)\bigr)$  such that  $u_n\to \overline{u}$ in $\mathcal{C}\bigl([0, T_0], \Sob^{-1}(E)\bigr)$ and $\overline{u}$ satisfies
                  \begin{equation*}
                    \overline{u}(t) = T(t) \overline{u_{0}} + \medint\int_{0}^{t} T(t-s) f(s) ds,
                  \end{equation*}
                  for all $t \in [0, T_{0}]$. Since $u_{n}(0)=u_{0}$ for all $n$, we deduce  $\overline{u_{0}}=u_0$. As $f\in L_2(0,T_0; H)$, in particular $\overline{u}\in \mathcal{C}\bigl([0,T_0];\operatorname{dom}(E_{-1})\bigr) = \mathcal{C}([0,T_0]; H)$.\qedhere
\end{enumerate}
\end{proof}

Before we move on to prove that the graph of $\Phi$ is weakly $\times$ weakly sequentially closed in $M \times \Leb^{2}(0,T_0;H) \times \Leb^{2}(0,T_0;\mathcal{H})$, we have to make a regularization argument involving the right-hand sides associated to the semigroup problem \eqref{eq:CauchyProblem}. For the convenience of the reader, we give a short recap of all convergences in the scenario that includes the convergence of the regularizations:
\begin{remark}[provisional summary]
  \label{th:Convergences}
When verifying the sequential weak $\times$ weak closedness of the graph of $\Phi$, we will have the following convergences:
\begin{alignat*}{5}
  \begin{aligned}
    (f_{n},g_{n})&\in \Leb^{2}(0,T_0;H)\times \Leb^{2}(0,T_0;\mathcal{H}) &&\text{with}\ (f_{n},g_{n})&\rightharpoonup\colon (f,g) &\quad\text{in}\ \Leb^{2},\\
    (\overline{f_{n}},\overline{g_{n}}) &\in \Sel F\bigl(S(f_{n},g_{n})\bigr) \times \Sel G\bigl(S(f_{n},g_{n})\bigr) &&\text{with}\ (\overline{f_{n}},\overline{g_{n}})&\rightharpoonup\colon (\overline{f},\overline{g}) &\quad\text{in}\ \Leb^{2}.
  \end{aligned}
\end{alignat*}
We apply a regularization procedure on the right-hand sides $f_{n}$ using the Yosida-approximations
\begin{equation*}
  f_{n,\lambda}\coloneq \lambda \mathcal{R}(\lambda, E)f_{n}\quad\text{and}\quad f_{\lambda}\coloneq \lambda \mathcal{R}(\lambda, E)f,
\end{equation*}
where $\mathcal{R}(\lambda, E)=(\lambda+E)^{-1}$ denotes the resolvent, which exists for large enough $\lambda >0$ appealing to standard semigrouph theory as $-E$ is a generator. Introducing \begin{equation*}
    \bigl(u_{n,\lambda}, v_{n,\lambda}\bigr) \coloneq S(f_{n,\lambda},g_{n}), \quad \bigl(u_{\lambda},v_{\lambda}\bigr)\coloneq S(f_{\lambda},g) \quad \text{and}\quad (u,v) \coloneq S(f,g),
\end{equation*}
we obtain the following convergences:
\begin{enumerate}[leftmargin = 6ex, label = (\roman*)]
  \item\label{I1} $\bigl(f_{n},g_{n}\bigr) \xrightharpoonup[n\to \infty]{} (f,g)$ in $\Leb^{2}(0,T_{0};H) \times \Leb^{2}(0,T_{0};\mathcal{H})$.
  \item\label{I2} $\bigl(\overline{f_{n}}, \overline{g_{n}}\bigr) \xrightharpoonup[n\to \infty]{} (\overline{f}, \overline{g})$ in $\Leb^{2}(0,T_{0};H) \times \Leb^{2}(0,T_{0};\mathcal{H})$.
  \item\label{I3} $u_{n} \xrightarrow[n\to \infty]{} u$ in $\mathcal{C}\bigl([0,T_{0}];\Sob^{-1}(E)\bigr)$ with $u \in \mathcal{C}\bigl([0,T_{0}];H\bigr)$ and\\
        $v_{n} \xrightarrow[n\to \infty]{} v$ in $\mathcal{C}\bigl([0,T_{0}];\mathcal{H}\bigr)$ with $v \in \mathcal{C}\bigl([0,T_{0}];\mathcal{H}\bigr)$.
  \item\label{I4} $f_{n,\lambda} \xrightharpoonup[n\to \infty]{} f_{\lambda}$ in $\Leb^{2}\bigl(0,T_{0};\operatorname{dom}(E)\bigr)$.
  \item\label{I5} $u_{n,\lambda} \xrightarrow[n\to \infty]{} u_{\lambda}$ in $\mathcal{C}\bigl([0,T_{0}];H\bigr)$ with $u_{\lambda} \in \mathcal{C}\bigl([0,T_{0}];\operatorname{dom}(E)\bigr)$ and\\
        $v_{n,\lambda} \xrightarrow[n\to \infty]{} v_{\lambda}$ in $\mathcal{C}\bigl([0,T_{0}];\mathcal{H}\bigr)$ with $v_{\lambda} \in \mathcal{C}\bigl([0,T_{0}];\mathcal{H}\bigr)$.
  \item\label{I6} $f_{\lambda} \xrightarrow[\lambda \to \infty]{} f$ in $\Leb^{2}(0,T_{0};H)$.
  \item\label{I7} $u_{\lambda} \xrightarrow[\lambda \to \infty]{} u$ in $\mathcal{C}\bigl([0,T_{0}];H\bigr)$ and\\
        $v_{\lambda} \xrightarrow[\lambda \to \infty]{} v$ in $\mathcal{C}\bigl([0,T_{0}];\mathcal{H}\bigr)$.
  \item\label{I8} $f_{n,\lambda}\xrightarrow[\lambda \to \infty]{} f_{n}$ in $\Leb^{2}(0,T_{0};H)$.
  \item\label{I9} $u_{n,\lambda} \xrightarrow[\lambda \to \infty]{} u_{n}$ in $\Leb^{2}(0,T_{0};H)$ and\\
        $v_{n,\lambda} \xrightarrow[\lambda \to \infty]{} v_{n}$ in $\mathcal{C}\bigl([0,T_{0}];\mathcal{H}\bigr)$.
\end{enumerate}
\end{remark}
\begin{proof}\phantom{.}
  \begin{itemize}[leftmargin=4ex]
  \item \Cref{I1} and \cref{I2} hold by assumption.
  \item \Cref{I3} holds due to \cref{th:ConvergenceOfSolutions}.
  \item \Cref{I4} holds because the Yosida-approximation regularizes the functions, in particular $f_{n,\lambda} \in \Leb^{2}\bigl(0,T_{m}; \operatorname{dom}(E)\bigr)$. The convergence is easily verified:
  For $\varphi \in \Leb^{2}\bigl(0,T_{m}; \Sob^{-1}(E)\bigr) = \Leb^{2}\bigl(0,T; \operatorname{dom}(E)\bigr)'$ holds:
  \begin{align*}
    \dualprod[\big]{\lambda \mathcal{R}(\lambda, E)f_{n}}{\varphi}
    &= \dualprod[\big]{f_{n}}{\underbrace{\lambda \mathcal{R}(\lambda, E')\varphi}_{\in \Leb^{2}(0,T;H')}}\\
    &\rightarrow \dualprod[\big]{f}{\lambda \mathcal{R}(\lambda, E')\varphi}
    = \dualprod[\big]{\lambda \mathcal{R}(\lambda, E)f}{\varphi}
  \end{align*}
  \item For \cref{I5} we first appeal to \cref{th:SemiGroupSolution} to see that $u_{n,\lambda} \in \mathcal{C}\bigl([0,T_{m}];\operatorname{dom}(E)\bigr)$. Then an application of \cref{th:ConvergenceOfSolutions} shows the first claim, where notably $H$ is replaced by $\operatorname{dom}(E)$.\\
  Appealing to \cref{th:MonotoneSolution}, we see that $v_{n,\lambda} \in \mathcal{C}\bigl([0,T_{0}];\mathcal{H}\bigr)$. An application of \cref{th:UniformConvergence} yields the second claim.
  \item \Cref{I6} and \cref{I8} are well-known properties of the Yosida-approximation.
  \item The first part of \cref{I7} follows from Duhamel's formula and \cref{I6}. The second part of \cref{I7} follows from \cref{th:UniformConvergence} again.
  \item For \cref{I9} we first verify convergence in $\Leb^{2}(0,T_{0};H)$:
  \begin{align*}
    \norm{u_{n,\lambda} - u_{n}}_{2}^{2}
    &= \medint\int_{0}^{T_{0}} \norm[\Big]{\medint\int_{0}^{t} T(t-s)\bigl[\lambda \mathcal{R}(\lambda,E)f_{n}-f_{n}\bigr](s) \dx[s]}^{2}\dx[t]\\
    &\leq \medint\int_{0}^{T_{0}} t \medint\int_{0}^{t} \norm[\big]{T(t-s)\bigl[\lambda \mathcal{R}(\lambda,E)f_{n}-f_{n}\bigr](s)}^{2} \dx[s]  \dx[t]\\
    &\leq \medint\int_{0}^{T_{0}} t \widetilde{C}^{2}\medint\int_{0}^{T_{0}} \norm[\big]{\bigl[\lambda \mathcal{R}(\lambda,E)f_{n}-f_{n}\bigr](s)}^{2} \dx[s]  \dx[t]\\
    &= \tfrac{T_{0}^{2}\widetilde{C}^{2}}{2} \norm{f_{n,\lambda} - f_{n}}_{\Leb^{2}(0,T_{0};H)}^{2}
  \end{align*}
  The second part follows immediately from the realization $\lambda \mathcal{R}(\lambda, E)f_{n} \xrightarrow[\lambda \to \infty]{} f_{n}$ and \cref{th:UniformConvergence} again.\qedhere
  \end{itemize}
\end{proof}

The missing ingredient is now to prove the existence of measurable selections
\begin{align*}
  \bigl(\overline{f_{n,\lambda}}, \overline{g_{n,\lambda}}\bigr) &\in \Sel F(u_{n,\lambda},v_{n}) \times \Sel G(u_{n,\lambda},v_{n})\\
  \intertext{suitably close to given (measurable) selections}
  \bigl(\overline{f_{n}}, \overline{g_{n}}\bigr) &\in \Sel F(u_{n},v_{n}) \times \Sel G(u_{n},v_{n})\text{,}
\end{align*}
for $(u_{n},v_{n}) = S(f_{n},g_{n})$ with $(f_{n},g_{n}) \in M$ from \cref{eq:M}. For that \cref{th:CloseSelections} will come into play. It is only  here (!), that the increased regularity assumption of Hausdorff-continuity and half weak continuity will be used. 
\begin{remark} In the proof of the following result we will have occasion to apply \cref{lem:weakstrong} (and the subsequent remark). To that end, note that $S_{\lambda} \coloneqq (\lambda+\lambda_0) \mathcal{R}(\lambda+\lambda_0,E)$ fits the assumptions needed there. Indeed, continuity is plain from the resolvent equation. Strong convergence of the adjoints follows from the fact that the adjoint of $E$ is the generator of the adjoint semigroup and thus also adheres to the principles and properties of the Yosida-approximation.
\end{remark}

\begin{lemma}
  \label{th:WeakSeqMeasurability}
  The graph of $\Phi$ from \eqref{eq:FixedPointMap} is weakly $\times$ weakly sequentially closed in $\Leb^{2}(0,T_{0};H)\times \Leb^{2}(0,T_{0};\mathcal{H})$.\footnote{We understand the term ``graph'' is ambiguous here, but it is this notion that is frequently applied in the literature, see, e.g.~\cite{Vrabie1995}.}
\end{lemma}
\begin{proof}
  Let $(f_{n},g_{n})\in M$; $(u_{n},v_{n}) \coloneqq S(f_{n},g_{n})$. Let $(\overline{f_{n}},\overline{g_{n}}) \in \Sel F(u_{n},v_{n}) \times \Sel G(u_{n},v_{n})$ with
  \begin{alignat*}{5}
  \begin{aligned}
    (f_{n},g_{n})\rightharpoonup (f,g) &\quad\text{in}\ \Leb^{2}(0,T_0;H)\times \Leb^{2}(0,T_0;\mathcal{H}), \\
   (\overline{f_{n}},\overline{g_{n}})\rightharpoonup (\overline{f},\overline{g}) &\quad\text{in}\ \Leb^{2}(0,T_0;H)\times \Leb^{2}(0,T_0;\mathcal{H}),
  \end{aligned}
\end{alignat*}
and $(u,v)\coloneqq S(f,g)$.

  Our aim is to show that $(\overline{f},\overline{g}) \in \Phi(f,g)= \Sel F(u,v) \times \Sel G(u,v)$. For this,
   we define $f_{n,\lambda}, f_{\lambda}, u_{n,\lambda}$ and $u_{\lambda}$ as in \cref{th:Convergences}. We aim to apply  \cref{th:CloseSelections} to construct sufficiently close selections for $\overline{f_n}$ and $\overline{g_n}$:

 To start off, we note that since $u_\lambda\to u$ uniformly, for all $\delta>0$, there is $\lambda_0>0$ such that for all $\lambda\geq \lambda_0$ we have
   \[
     \sup_{t\in [0,T_0]} \Hd \big(F(u_\lambda(t),v(t)),F(u(t),v(t))\big)\leq \delta.
   \]
   Moreover, by \eqref{eq:need}, we find $\lambda_1\geq \lambda_0$ such that
   \[
\forall \lambda \!\geq \!\lambda_1, n \!\in \!\N, t \!\in \![0,T_0]\colon \Hd\big(F(u_{n,\lambda}(t),v_n(t)),F(u_n(t),v_n(t))\big)\leq \delta.
\]

   Next, let $\varepsilon>0$. By virtue of \cref{th:CloseSelections} (see also \cref{rem:actualthm3.3.1}) and by choosing $\delta>0$ accordingly (and correspondingly $\lambda_1\geq\lambda_0>0$), we find  $\overline{f_{n,\lambda_1}}\in \Sel F(u_{n,\lambda_1},v)$ with
   \[
       \|\overline{f_{n,\lambda_1}} - \overline{f_n}\|\leq \varepsilon.
   \] 
   By \cref{th:Convergences} \cref{I5}, $u_{n,\lambda_1}\to u_{\lambda_1}$ in $\mathcal{C}([0,T_0];H)$. Hence, by \cref{th:CloseSelections}  (see also \cref{rem:actualthm3.3.1}), we find $n_0\in \N$ such that for all $n\geq n_0$ there is $\overline{f_{\lambda_1}^{(n)}} \in \Sel F(u_{\lambda_1},v)$ such that
   \begin{equation*}
    \norm[\big]{\overline{f_{n,\lambda_1}} - \overline{f_{\lambda_1}^{(n)}}}_{\Leb^{2}(0,T_{0};H)}\leq \varepsilon.
  \end{equation*}
  From yet another application of \cref{th:CloseSelections} and \cref{rem:actualthm3.3.1} it follows that, for all $n\in \N$, we find $\overline{f^{(n)}}\in \Sel F(u,v)$ such that
  \begin{equation*}
    \norm[\big]{\overline{f^{(n)}} - \overline{f_{\lambda_1}^{(n)}}}_{\Leb^{2}(0,T_{0};H)}\leq \varepsilon.
  \end{equation*}

  To conclude, we can assume without loss of generality that all considered sequences (in $n$) converge weakly as $n\to\infty$. Then $\widetilde{f} \coloneqq \lim_{n\to\infty} \overline{f^{(n)}} \in  \Sel F(u,v) $ as $\Sel F(u,v) $ is (weakly) closed, by \cref{rem:wclosed}. Moreover, by the liminf-inequality for norms of weak limits, it follows that
  \[
     \| \widetilde{f} - \overline{f}\|\leq 3\varepsilon.
  \]
  As $\varepsilon>0$ was arbitrary, it follows that 
  \[\overline{f} \in \bigcap_{\varepsilon>0} \big(\Sel F(u,v)+B_{H}[0,\varepsilon]\big) =\overline{\Sel F(u,v)} = \Sel F(u,v). \]
  Quite similarly (but easier) it follows that $\overline{g} \in \Sel G(u,v)$.
\end{proof}

The proof of our existence theorem is now a simple collection of statements:
\begin{proof}[Proof of \cref{th:LocalExist}]
  The fixed point map $\Phi$ from \eqref{eq:FixedPointMap} is well-defined according to \cref{th:PhiSelfmap} and \ref{th:PropertiesOfPhi}. \Cref{th:PropertiesOfPhi} additionally shows that $\Phi$ is bcc. \Cref{th:WeakSeqMeasurability} finally shows that the graph of $\Phi$ is weakly $\times$ weakly sequentially closed in $M\times \Leb^{2}(0,T_{0};H) \times \Leb^{2}(0,T_{0};\mathcal{H})$.\\
  Applying the multivalued fixed point \cref{th:FixedPointTheorem}, we can conclude the existence of a fixed point in $M$, that is, by construction, a pair $(f,g) \in M$ such that $(f,g) \in \Phi(f,g)$ and consequently $(u,v) = S(f,g)$  is one weak solution of the system \eqref{eq:SingleValuedProblem}.
  By construction (see \eqref{eq:M}),
  \begin{equation*}
    M \subseteq \Leb^{2}(0, T_0; H) \times \Leb^{2}(0, T_0; \mathcal{H}).
  \end{equation*}
  \cite[thm.~3.6]{Brezis1973} ensures that  $(u,v) = S(f,g)$ is a strong solution of \eqref{eq:SingleValuedProblem}
  and the claim follows.\qedhere
\end{proof}

%%%%%%%%%%%%%%%%%%%%%%%%%%%%%%
%    S E C T I O N
\subsection{Global existence}
\label{sec:GlobalExist}

An investigation of the proof of \cref{th:LocalExist} (or carefully reading \cref{sec:LocalExist}) reveals that the proof holds up for any right-hand sides $(f,g)$ that remain bounded in $\Leb^{2}(0,T;H)$-norm and $\Leb^{2}(0,T;\mathcal{H})$-norm, i.e., selections that remain bounded as $T$ varies give rise to solutions $(u,v) \in \mathcal{C}\bigl([0,T];H\bigr)\times \mathcal{C}\bigl([0,T];\mathcal{H}\bigr)$. Hence, if the maximal existence interval $[0,T_{\mathrm{max}})$ is finite, then necessarily,
\begin{equation}
  \label{eq:Blowup}
  \lim_{t\uparrow T_{\max}}\norm{u(t)}_{H} = \infty \qquad \text{or} \qquad \lim_{t\uparrow T_{\max}}\norm{v(t)}_{\mathcal{H}} = \infty\text{.}
\end{equation}
Conversely, such a blowup prevents a continuable solution by contradiction. Consequently, we have the following theorem:
\begin{theorem}[maximal existence interval]
  \label{th:MaxExist}
  Under the assumptions of \cref{th:LocalExist} the solution $(u,v)$ provided by \cref{th:LocalExist} exists up to a maximal time $T_{\max}\in (0,\infty ]$. If $T_{
  \max}$ is finite, \eqref{eq:Blowup} holds.
\end{theorem}
We point out, that \cref{th:MaxExist} for the case $T=\infty$ does only state that $u,v$ are continuous functions on the half-line and says nothing about (square-) integrability of $u,v$. We now move on to give a criterion for global existence of solutions. Following the standard route in solution theory for differential equations, this can be done by prescribing a linear growth bound for the right-hand sides and making use of Gr\"onwall's inequality.
\begin{theorem}[global existence]
  \label{th:GlobalExist}
  Under the assumptions of \cref{th:LocalExist}, let
  additionally $F$ and $G$ satisfy a linear growth bound. Then the solution $(u,v)$ provided by \cref{th:LocalExist} is global.\footnote{i.e., the strong solution $(u,v)$ of \eqref{eq:IVP} is defined on $[0, T_m]$.}
\end{theorem}
\begin{proof}
  By virtue of unique solvability from \cref{th:LocalExist}, we can extend the local solution $(u,v)$ to a maximal existence time $T_{\mathrm{max}}$. We show that $T_{\mathrm{max}}=T$, i.e., the solution is global, by proving an exponential growth bound. Let
  \begin{equation*}
    \bigl(f(t),g(t)\bigr) \in F\bigl(u(t),v(t)\bigr) \times G\bigl(u(t),v(t)\bigr) \quad \text{for a.e.}\ 0\leq t < T_{\mathrm{max}}\text{.}
  \end{equation*}
  We have to estimate the solution $(u,v)$ to problem \eqref{eq:IVP}:
  \begin{enumerate}[label = (\roman*), leftmargin = 5ex]
    \item We start with the monotonous part:
    
          The equation
          \begin{equation*}
            \tfrac{\dd}{\dx[t]}v(t) + A(t)v(t) = g(t)
          \end{equation*}
          holds as an equality in $\mathcal{H}$, appealing to \cref{th:MonotoneSolution} for almost all $0 \leq t < T_{\mathrm{max}}$, because $v$ is absolutely continuous on any $[\delta, T_{\mathrm{max}}]$, $0<\delta \leq T_{\mathrm{max}}$, and $v(t) \in \mathcal{D}$ for almost all $t \in [0,T_{\mathrm{max}})$ (appealing to \cref{th:MonotoneSolution}). Hence, we can test with $v(t)$ and integrate from $0$ to $0\leq t < T_{\mathrm{max}}$ to obtain
          \begin{equation*}
            \dualprod[\big]{\tfrac{\dd}{\dx[t]}v(t)}{v(t)} + \dualprod{A(t)v(t)}{v(t)} = \dualprod{g(t)}{v(t)}\text{.}
          \end{equation*}
          Consequently, using monotonicity of $A(t)$, we obtain
          \begin{equation}
            \tag{$\star$}
            \tfrac{1}{2} \norm{v(t)}_{\mathcal{H}}^{2} \leq \tfrac{1}{2}\norm{v_0}_{\mathcal{H}}^{2} + \medint\int_{0}^t \dualprod{g(s)}{v(s)}_{\mathcal{H},\mathcal{H}}\dx[s].
          \end{equation}
          Appealing to the linear growth bound of $G$, we have parameters $a,b,c>0$ as in \cref{def:LinearGrowth} s.t.\
          \begin{equation*}
            \norm{g(s)}_{\mathcal{H}} \leq a \norm{u(s)}_{H} + b\norm{v(s)}_{\mathcal{H}} + c.
          \end{equation*}
          Now we estimate using ($\star$):
          \begin{align*}
            \norm{v(t)}_{\mathcal{H}}^2
            &\leq \norm{v_0}^2 + 2 \medint\int_{0}^{t} \dualprod{g(s)}{v(s)}\dx[s] \\
            &\leq \norm{v_0}^2 + 2 \medint\int_{0}^{t}\bigl[a \norm{u(s)}_{H} + b \norm{v(s)}_{\mathcal{H}} + c\bigr] \norm{v(s)}_{\mathcal{H}} \dx[s].\\
            \intertext{Consequently, since $v \in \mathcal{C}([0,t];\mathcal{H})$, we have}
            \norm{v}_{\mathcal{C}([0,t];\mathcal{H})}^{2} &\leq \norm{v_0}^2 + 2 \medint\int_{0}^{t}\bigl[a \norm{u(s)}_{H} + b \norm{v(s)}_{\mathcal{H}} + c\bigr] \norm{v(s)}_{\mathcal{H}} \dx[s]\\
            &\leq \norm{v_0}^2 + 2 \norm{v}_{\mathcal{C}([0,t];H)}\medint\int_{0}^{t}\bigl[a \norm{u(s)}_{H} + b \norm{v(s)}_{\mathcal{H}} + c\bigr] \dx[s].\\
            \intertext{An application of $2xy\leq \tfrac{1}{2}x^2+2y^2$, $x,y\in \R$, shows}
            &\leq \norm{v_0}^2 + \tfrac{1}{2}\norm{v}^{2}_{\mathcal{C}([0,t];H)} + 2\Bigl(\medint\int_{0}^{t}\bigl[a \norm{u(s)}_{H} + b \norm{v(s)}_{\mathcal{H}} + c\bigr] \dx[s]\Bigr)^{2}\text{.}\\
            \intertext{Consequently, we obtain}
            \qquad\norm{v}_{\mathcal{C}([0,t];\mathcal{H})}^{2} &\leq 2\norm{v_0}^2 + 4\Bigl(\medint\int_{0}^{t}\bigl[a \norm{u(s)}_{H} + b \norm{v(s)}_{\mathcal{H}} + c\bigr] \dx[s]\Bigr)^{2}\text{.}
            \intertext{The simple inequality $\sqrt{a^{2}+b^{2}+c^{2}}\leq a + b +c$ for nonnegative $a,b,c$ implies}
            \norm{v}_{\mathcal{C}([0,t];\mathcal{H})} &\leq \sqrt{2}\norm{v_{0}} + 2cT_{\max} + 2\medint\int_{0}^t a \norm{u(s)}_{H} + b \norm{v(s)}_{\mathcal{H}}\dx[s].
          \end{align*}
    \item We continue with the semigroup part:
    
          Using the definition of a mild solution, we see for some $C\geq 0$ and all $0\leq t\leq T_{\mathrm{max}}$,
          \begin{equation*}
            \norm{u(t)}_{H} \leq C \norm{u_0}_{H} + C\medint\int_{0}^{t} \norm{f(s)}_{H}\dx[s].
          \end{equation*}
          Using the linear growth bound for $F$ and the fact that $f \in F(u,v)$, we can estimate
          \begin{equation*}
            \norm{f(s)}_{H} \leq a \norm{u(s)}_{H} + b \norm{v(s)}_{\mathcal{H}} + c
          \end{equation*}
          for some $a,b,c>0$ and we obtain $\widetilde{C}>0$ independent of $t$ s.t.
          \begin{equation*}
            \norm{u(t)}_{H} \leq \widetilde{C} + C\medint\int_{0}^t a \norm{u(s)}_{H} + b \norm{v(s)}_{\mathcal{H}}\dx[s].
          \end{equation*}
          It follows that $\norm{u}_{\mathcal{C}([0,t];H))} \leq \widetilde{C} + C\medint\int_{0}^t a \norm{u(s)}_{H} + b \norm{v(s)}_{\mathcal{H}}\dx[s]$.
  \end{enumerate}
  Both of these estimates together yield for some $K(T_{\mathrm{max}}),\rho\geq 0$
  \begin{equation*}
    \norm{u}_{\mathcal{C}([0,t];H)} + \norm{v}_{\mathcal{C}([0,t];\mathcal{H})} \leq K(T_{\mathrm{max}}) + \rho \medint\int_{0}^t [\norm{u(s)}_{H} + \norm{v(s)}_{\mathcal{H}}]\dx[s]\text{.}
  \end{equation*}
  Gr\"onwall's lemma thus concludes
  \begin{equation*}
    \norm{u}_{\mathcal{C}([0,t];H)} + \norm{v}_{\mathcal{C}([0,t];\mathcal{H})} \leq K(T_{\mathrm{max}}) e^{\rho t}.
  \end{equation*}
  This proves the claim.
\end{proof}

% STYLE
\vspace{1cm}

\begin{remark}\phantom{}
  \begin{enumerate}[label = (\roman*), leftmargin = 5ex]
          \item The assumptions of \cref{th:GlobalExist} can be weakened as far as the monotone part of problem \eqref{eq:IVP} is concerned. There, it suffices to assume the notion of ``positive sublinearity'' (see~\cite[def.~3.2.5]{Vrabie1995}) of $G$. The proof is done by the usual means, see~the proof of \cite[thm.~41]{SimsenWittbold2019} for the standard strategy. Since positive sublinearity seems to be of purely theoretical relevance, we made use of the more straightforward concept of a linear growth bound.
    
    \item The assumption on $F$ cannot be easily lowered to ``positive sublinear'' though as $E$ is not assumed to be maximal accretive. However, even then the argument does not become elementary: For the sake of the argument additionally assume that $E$ is maximal accretive.
    
          To test the equation $\tfrac{\dd}{\dx[t]}u(t) + E u(t) = g(t)$ with $u(t)$ we need to go to the level of the extrapolation space $\Sob^{-1}(E)$ again, where we can estimate
          \begin{align*}
            \tfrac{\dd}{\dx[t]}\norm{u(t)}_{\Sob^{-1}(E)}^{2}
            &= 2 \Re \dualprod{\tfrac{\dd}{\dx[t]}u(t)}{u(t)}_{\Sob^{-1}(E)}\\
            &= 2 \Re \underbrace{- \dualprod{Eu(t)}{u(t)}_{\Sob^{-1}(E)}}_{\leq 0} + 2\Re \dualprod{u(t)}{g(t)}_{\Sob^{-1}(E)}\text{,}
          \end{align*}
          where we use accretivity of $E_{-1}$. We infer
          \begin{equation*}
            \norm{u(t)}_{\Sob^{-1}(E)}^{2} \leq \norm{u_0}_{\Sob^{-1}(E)}^{2} + 2 \medint\int_{0}^{t} \dualprod{u(s)}{f(s)}_{\Sob^{-1}(E)}\dx[s]\text{.}
          \end{equation*}
          This estimate provides an opportunity to apply the notion of positive sublinearity. The problem however is that different to the proof of \cref{th:LocalExist}, where an extrapolation argument was used as well, here we cannot recover an estimate of $u(t)$ in $\norm{\argdot}_{H}$-norm. Consequently, the entire theorem would need to be formulated for $\Sob^{-1}(E)$ instead of $H$.\\
          This is possible, including local existence, by verbatim arguments as in the proof of \cref{th:LocalExist}, because $\Sob^{-1}(E)$ is a separable Hilbert space.
          
          The issue however is that problems of practical relevance are formulated for right-hand sides in $H$, e.g., for $H=\Leb^{2}(\Omega)$ and not the extrapolation space. Checking the corresponding norm estimates (for the $\Sob^{-1}(E)$ case) is also an impractical endeavor in practice. In any case, one needs to require a regularising $F$ in some way; otherwise one would run into similar problems, when proving weak$\times$weak-sequential closedness with $H=\Sob^{-1}(E)$.
    \item The case $T = \infty$ is allowed in \cref{th:GlobalExist}, but requires the semigroup generated by $E$ to be bounded. In that case, one obtains an estimate
          \begin{equation*}
            \norm{u}_{\mathcal{C}([0,t];H)} + \norm{v}_{\mathcal{C}([0,t];\mathcal{H})} \leq \hat{C}t e^{\rho t}
          \end{equation*}
          in the last line of the proof, where $\hat{C}$ is some constant.
  \end{enumerate}
\end{remark}

%%%%%%%%%%%%%%%%%%%%%%%%%%%%%%
%    S E C T I O N
\section{Examples}
\label{sec:Examples}

We first give a short example of a class of forcing terms $F$ and $G$ befitting our results. After that, we will present maximal monotone operators $A(\argdot)$ that satisfy \cref{ass:D} and stem from a generalized Schr\"odinger--Debye system. At the tail end of this section we outline several classical examples of partial differential equations that can be posed as Cauchy problems befitting the assumptions of \cref{sec:LocalExist} and \cref{sec:GlobalExist}.

%%%%%%%%%%%%%%%%%%%%%%%%%%%%%%
%    S U B S E C T I O N
\subsection{Suitable right-hand sides}
\label{subsec:RHS} Let $X$ be a separable Hilbert space and let $(e_{n})_{n \in \N}$ in $X$ be an ONB. For $n \in \N$ let $\varphi_{n}\colon H\times \mathcal{H} \to \mathbb{R}$ be continuous, bounded with $\bigl(\norm{\varphi_{n}}_{\infty}\bigr)_{n} \in \ell^{2}(\N)$ and such that $\varphi_n(\cdot, y)$ is weakly continuous on bounded sets for all $y\in \mathcal{H}$. Then define
\begin{equation*}
  F(u,v) \coloneq \overline{\mathrm{co}}\bigl\{\varphi_{n}(u,v)e_{n}; n \in \N\bigr\}\text{,}
\end{equation*}
which are by definition nonempty, closed and convex sets. Bessel's inequality shows that the sets are bounded as well. First we verify Hausdorff-continuity of $F\colon H\times \mathcal{H} \to \mathcal{P}(X)$.

Let $\epsilon > 0$ and $N \in \N$ be such that $\sum_{k=N+1}^{\infty}\norm{\varphi_{k}}_{\infty}^{2}<\epsilon$. Let $(u,v) \in H\times \mathcal{H}$. Appealing to continuity of $\varphi_{n}$, there exist open neighbourhoods $U_{n},V_{n}$ of $u$ and $v$ respectively s.t.\
\begin{equation*}
  \forall 1\!\leq \!n \!\leq \!N\,\,\forall \bigl(\overline{u},\overline{v}\bigr) \in U_{n}\times V_{n}\colon \abs{\varphi_{n}(\overline{u},\overline{v}) - \varphi(u,v)} < \epsilon\text{.}
\end{equation*}
Since
\begin{equation*}
  U\coloneq \bigcap_{1\leq n \leq N}U_{n} \qquad\text{and}\qquad V\coloneq \bigcap_{1\leq n \leq N}V_{n}
\end{equation*}
are open neighbourhoods of $u$ and $v$, we proceed to estimate $\Hd\bigl(F(u,v),F(\overline{u},\overline{v})\bigr)$ for all $(\overline{u},\overline{v}) \in U\times V$. To that end, let $x \in F(u,v)$ and write $x = \sum_{k=1}^{\infty}\lambda_{k}\varphi_{k}(u,v)e_{k}$ for some convex combination $\sum_{k=1}^{\infty}\lambda_{k} = 1$. We compute
\begin{align*}
  &d\big(x,F(\overline{u},\overline{v})\big)
  = \inf \biggl\{\norm[\bigg]{x - \sum_{k=1}^{\infty}\mu_{k}\varphi_{k}(\overline{u},\overline{v})e_{k}}; \sum_{k=1}^{\infty}\mu_{k} =1 \,\text{convex combination}\biggr\}\\
  &\leq \inf \biggl\{\Bigl(\sum_{k=1}^{\infty} \underbrace{\abs{\lambda_{k}-\mu_{k}}^{2}}_{\leq \lambda_{k}}
    \abs{\varphi_{k}(u,v) - \varphi_{k}(\overline{u},\overline{v})}^{2}
    \underbrace{\norm{e_{k}}^{2}}_{=1}\Bigr)^{\nicefrac{1}{2}};
    \sum_{k=1}^{\infty}\mu_{k} =1 \,\text{conv. comb.}\biggr\}\\
  &\leq \Bigl(\sum_{k=1}^{N} \lambda_{k}\left.\underbrace{\abs{\varphi_{k}(u,v) - \varphi_{k}(\overline{u},\overline{v})}}_{\leq \epsilon}\right.^{2}\Bigr)^{\nicefrac{1}{2}}
    + \Bigl(\underbrace{\sum_{k=N+1}^{\infty}\abs{\varphi_{k}(u,v) - \varphi_{k}(\overline{u},\overline{v})}^{2}}_{\leq \epsilon^{2}}\Bigr)^{\nicefrac{1}{2}}\\[-1ex]
  &\leq \epsilon + \epsilon\text{,}
\end{align*}
where $N$ is an index chosen suitably large enough. A similar calculation shows for $\overline{x} \in F(\overline{u},\overline{v})$,
\begin{equation*}
  d(\overline{x},F(u,v)) \leq 2\epsilon\text{.}
\end{equation*}
Consequently, 
\begin{align*}
  \Hd\bigl(F(u,v),F(\overline{u},\overline{v})\bigr)
  &= \max \biggl\{\sup_{x \in F(u,v)}d\bigl(x,F(\overline{u},\overline{v})\bigr), \sup_{\overline{x} \in F(u,v)}d\bigl(\overline{x},F(u,v)\bigr)\biggr\}\leq  2\epsilon
\end{align*}
Hence, $F$ is Hausdorff-continuous. In order to show that $F$ is  half weakly continuous, choose correspondingly different neighbourhoods $U_n$ for fixed second arguments in the above and proceed similarly.

We now comment on our assumptions for global solutions:\\
For global solutions, we demand a linear growth bound (see~\cref{def:LinearGrowth}) in \cref{th:GlobalExist}. If we demand that $F\colon H \times \mathcal{H}\to \mathcal{P}(X)$ be of the form as above and
\begin{equation*}
\forall (u,v)\in H\times \mathcal{H}\colon  \varphi_{k}(u,v) = c_{k}\dualprod{u}{e_{k}} + \nu_{k} (v)\norm{v}\text{,}
\end{equation*}
where $\nu_{k} \colon H\to \mathbb{R}$ is continuous, bounded and satisfies $\bigl(\norm{\nu_{k}}_{\infty}\bigr)_{k\in \N} \in \ell^{2}(\N)$ and $c_{k}\in \mathbb{R}$ is a constant with $(c_{k})_{k\in\N}\in \ell^{2}(\N)$, then $F$ satisfies a linear growth bound. We note in passing that this class of examples also satisfies the half weak continuity requirement.

For $x \in F(u,v)$ we can estimate
\begin{align*}
  \norm{x}^{2} &= \sum_{k=1}^{\infty}\lambda_{k}^{2}\varphi_{k}(u,v)^{2}\underbrace{\norm{e_{k}}^{2}}_{=1}\\
               &= \sum_{k=1}^{\infty}\lambda_{k}^{2}\bigl(c_{k}\dualprod{u}{e_{k}} + \nu_{k} (v)\norm{v}\bigr)^{2}\\
               &\leq 2 \sum_{k=1}^{\infty}\lambda_{k}^{2}c_{k}^{2}\dualprod{u}{e_{k}}^{2} + \nu_{k} (v)^{2}\norm{v}^{2}\\
               &\leq 2\sum_{k=1}^{\infty}c_{k}^{2} \norm{u}^{2} + \sum_{k=1}^{\infty}\nu_{k}(v)^{2}\norm{v}^{2}\\
               &= 2 \norm[\big]{(c_{k})_{k}}_{2}^{2} \norm{u}^{2} + 2\norm[\big]{(\norm{\nu_{k}}_{\infty})_{k}}_{2}^{2}\norm{v}^{2}\text{,}
\end{align*}
which proves the linear growth bound.

%%%%%%%%%%%%%%%%%%%%%%%%%%%%%%
\subsection{A time-dependent maximal monotone operator for the Schr\"odinger--Debye system}
\label{subsec:SchroedingerDebye}
We first briefly recapitulate Lebesgue--Sobolev spaces with variable exponents.

Consider $\Omega \subseteq \mathbb{R}^{d}$, $d \geq 1$, a bounded smooth domain. Let $p\colon \Omega \to [1, \infty)$ be measurable with $\mathrm{ess}\sup_{x\in \Omega}p(x) < \infty$, i.e., a {\em bounded variable exponent}. The corresponding Lebesgue space, $\Leb^{p(x)}(
\Omega)$, is defined as the class of those Lebesgue-measurable functions $f\colon \Omega \to \R^{d}$ satisfying
\begin{equation*}
  \int_{\Omega} \abs{f(x)}^{p(x)}\dx[x] < \infty\text{,}
\end{equation*}
the corresponding norms are given by
\begin{equation*}
  \norm{f}_{p(x)} \coloneq \inf \Bigl\{\lambda > 0; \medint\int_{\Omega}\abs[\big]{\tfrac{f(x)}{\lambda}}^{p(x)}\dx[x]\leq 1\Bigr\}\text{.}
\end{equation*}
For details on and properties of $\Leb^{p(x)}(\Omega)$, we refer to \cite[ch.~3]{Ruzicka2011}. The Sobolev spaces $\mathrm{W}^{k,p(x)}(\Omega)$, $k\in \N$, consist of those functions $f \in \Leb^{p(x)}(\Omega)$ s.t.\ all distributional derivatives of order $\leq k$ are again in $\Leb^{p(x)}(\Omega)$, equipped with the norm
\begin{equation*}
  \norm{f}_{k,p(x)} \coloneq \inf \Bigl\{\lambda > 0; \forall \alpha \in \N_{0}^{d}, \abs{\alpha}\leq k\colon \medint\int_{\Omega}\abs[\big]{\tfrac{\partial^{\alpha}f(x)}{\lambda}}^{p(x)}\dx[x]\leq 1\Bigr\}\text{.}
\end{equation*}
For details and properties of $\mathrm{W}^{k,p(x)}(\Omega)$ we refer to \cite[ch.~8]{Ruzicka2011}.

To define our example, let $\mathcal{H} = \Leb^{2}(\Omega)$ and let $p \in \mathcal{C}(\overline{\Omega}; \mathbb{R})$ satisfy
\begin{equation}
  \label{eq:p}
  p^+ \coloneq \max_{(x)\in \overline{\Omega}}  p(x) \geq \min_{(x)\in \overline{\Omega}}  p(x) \eqcolon p^{-} > 2.
\end{equation}
Then
\begin{equation*}
  \mathrm{W}^{1,p(\cdot)}(\Omega) \subseteq \mathcal{H} \subseteq \bigr(\mathrm{W}^{1,p(\cdot)}(\Omega)\bigl)'
\end{equation*}
with continuous and dense embeddings:
\begin{itemize}[leftmargin = 5ex]
  \item Appealing to \cite[lem.~8.1.8]{Ruzicka2011}, since $\Omega$ is assumed to be bounded, $\mathrm{W}^{1,p(x)}(\Omega)\hookrightarrow \mathrm{W}^{1,p^{-}}(\Omega)$ and by \cite[thm.~9.1.7]{Ruzicka2011}, since $\Omega$ is bounded and smooth, $\mathcal{C}^{\infty}(\overline{\Omega})$ is dense in $\mathrm{W}^{1,p(\argdot)}(\Omega)$. By the same argument, $\mathcal{C}^{\infty}(\overline{\Omega})$ is dense in $\mathrm{W}^{1,p^{-}}(\Omega)$.
  \item The continuity and density of the embedding $\mathrm{W}^{1,p^{-}}(\Omega)\hookrightarrow \mathrm{W}^{1,2}(\Omega)$ is well-known.
  \item Lastly, the compact, dense embedding $\mathrm{W}^{1,2}(\Omega)\hookrightarrow \Leb^{2}(\Omega)$ finishes the first inclusion.
  \item The second inclusion is proven by similar arguments using duality.
\end{itemize}
We will make use of coefficients $D(x,t)$ satisfying the following:
\begin{assumptions}
  \label{ass:D}
  Let $p \in \mathcal{C}(\overline{\Omega}; \mathbb{R})$ satisfy \eqref{eq:p}. Let $D \in \Leb^{\infty}\bigl([0,T]\times\Omega; \mathbb{R}\bigr)$ satisfy
  \begin{enumerate}[leftmargin=6ex, label = (D\arabic*)]
    \item $\exists\beta \!> \!0$ s.t.\ $0 < \beta \leq D(t,x)$ for almost all $(t,x) \in [0,T]\times\Omega$ and
    \item $D(t,x) \geq D(s,x)$ for each $x \in \Omega$ and $0 \leq t \leq s \leq T$.
  \end{enumerate}
\end{assumptions}
We consider the following operator
\begin{align*}
  A_{1}(t) \colon \mathrm{W}^{1,p(\cdot)}(\Omega) &\to \bigl(\mathrm{W}^{1,p(\cdot)}(\Omega)\bigr)',\\
  \dualprod{A_{1}(t)v}{w} &\coloneq \medint\int_{\Omega}D(t,x)\abs{\nabla v(x)}^{p(x)-2}\nabla  v(x)\nabla w(x)\dx[x]\\
  &\quad +  \medint\int_{\Omega}\abs{v(x)}^{p(x)-2}v(x)w(x)\dx[x].
\end{align*}
This modified $p$-Laplacian is well-defined and enjoys the following properties:

\begin{lemma}[{\cite[lem.~2.4 \& rem.~2.5]{Simsen2023}}]
  The operator $A_1(t) \colon \mathrm{W}^{1,p(\cdot)}(\Omega) \rightarrow \bigl(\mathrm{W}^{1,p(\cdot)}(\Omega)\bigr)'$ is monotone, coercive and hemicontinuous, for each $t\in (0,T]$.
\end{lemma}

To see that $A_{1}$ admits a realization $A$ of subdifferential type in $\mathcal{H} = \Leb^{2}(\Omega)$, we appeal to the following result:

\begin{theorem}[{\cite[thm.~2.8]{Simsen2023}}]
  \label{th:PropertiesOfPotential}
  The realization $A(t)$ of $A_1(t)$ in $\mathcal{H}$ is the subdifferential $\partial\phi^{D(t,\cdot)}_{p(\cdot)}$ of the following convex, lower semicontinuous map with non-trivial effective domain:
  \begin{align*}
    \phi^t &\coloneq \phi^{D(t,\cdot)}_{p(\cdot)} \colon \Leb^2(\Omega) \rightarrow \mathbb{R}\cup\{+\infty\}\\
    v &\mapsto \begin{cases}
      \medint\int_{\Omega}\tfrac{D(t,x)}{p(x)}\abs{\nabla v}^{p(x)}\dx[x] +\medint\int_{\Omega}\tfrac{1}{p(x)}\abs{v}^{p(x)}\dx[x] &\textrm{if}\ v\in \mathrm{W}^{1,p(\cdot)}(\Omega)\\
      +\infty &\textrm{otherwise.}
    \end{cases}
  \end{align*}
\end{theorem}

\

Consequently, the realization $A(t)$ of $A_1(t)$ in $\mathcal{H}=L^{2}(\Omega)$,
\begin{align*}
  A(t) \colon L^{2}(\Omega) \supseteq \operatorname{dom}\bigl(A(t)\bigr)&\to \Leb^{2}(\Omega)\\
  v&\mapsto -\operatorname{div}\bigl(D(t,x)\abs{\nabla v}^{p(x)-2}\nabla v\bigr) + \abs{v}^{p(x)-2}v\text{,}
\end{align*}
is a maximal monotone time-dependent operator densely defined on $\Leb^{2}(\Omega)$ with       
\begin{equation*}
  \operatorname{dom}\bigl(A(t)\bigr) \coloneq \bigl\{v \in \mathrm{W}^{1,p(x)}(\Omega)\colon A(t)v \in \Leb^{2}(\Omega)\bigr\}\text{.}
\end{equation*}
To verify that under \cref{ass:D}, the operators $A(t)$ satisfy \cref{ass:A}, which is the critical criterion for well-posedness of the evolution problem
\begin{equation}
  \label{eq:MonotoneEvProblem}
  \left\{
  \begin{aligned}
    \tfrac{\dd}{\dx[t]}v(t) + \partial \phi^{D(t,\cdot)}_{p(\cdot)} v(t) &=f(t) \qquad t>0\text{,}\\
    v(0)&=v_{0}\text{,}
  \end{aligned}
  \right.
\end{equation}
for some $v_{0}\in \mathcal{H}$, we make the following observations, following \cite[p.~2548f]{SimsenKloeden2014}.
\begin{remark}[$A(\argdot)$ satisfies \cref{ass:A}]
  \label{th:AssumptionsSatisfied}
  We comment on the items of \cref{ass:A} one by one:
  \begin{enumerate}[leftmargin = 5ex, label = (\roman*)]
    \item is satisfied with the nullset $Z$ being the empty set, appealing to \cref{th:PropertiesOfPotential}, where we additionally observe
          \begin{itemize}[leftmargin = 3ex]
            \item $\operatorname{dom}(\phi^{t}) \equiv \mathrm{W}^{1,p(x)}(\Omega)$ (dense in $\mathcal{H} = \Leb^{2}(\Omega)$ as argued above).
            \item $\partial\phi^{t}$ is single-valued, because the norm on $\mathcal{H}=\Leb^{2}(\Omega)$ is strictly convex.
            \item $\partial \phi^{t}(0) = 0$ for all $t>0$ (easily verified).
          \end{itemize}
    \item we verify by defining for $n \in \N_{\geq 1}$
          \begin{itemize}[leftmargin = 3ex]
            \item $K_{n}\coloneq n \,\,(>0)$,
            \item $g_{n}\colon [0,T]\to \mathbb{R},\, t\mapsto t+n$ (absolutely continuous with $g_{n}'\in \Leb^{\infty}(0,T)$),
            \item $h_{n}\colon [0,T]\to \mathbb{R},\, t\mapsto n$ (of bounded variation).
          \end{itemize}
          Then for $t \in [0,T]$, $w \in \operatorname{dom}(\phi^{t})$ with $\norm{w}\leq n$ and $s \in [t,T]$ let $\widetilde{w} \coloneq w \in X = \operatorname{dom}(\phi^{s})$. In the proof of \cite[thm.~3.5]{Simsen2023} for the choice $\alpha \coloneq \nicefrac{1}{2}$ and $\beta = 2$ it is verified that
          \begin{equation*}
            \phi^{s}(\widetilde{w}) = \phi^{s}(w) \leq \phi^{t}(w)\text{,}
          \end{equation*}
          where one uses item $(ii)$ from \cref{ass:D}.
  \end{enumerate}
  This makes sure \cref{ass:A} are satisfied for our particular choice of $A(\argdot)$.
\end{remark}

We can now couple the evolution operator $A(\cdot)$ with any of the semigroup generators discussed in the following subsections to obtain a starting problem \eqref{eq:IVP} of choice.

%%%%%%%%%%%%%%%%%%%%%%%%%%%%%%
%    S U B S E C T I O N
\subsection{Heat semigroup}
\label{subsec:Heat}
In our first example for possible semigroup generators, we plug the heat equation into the first equation in \eqref{eq:IVP}, i.e., we arrive at
\begin{equation}
  \label{eq:Heat}
  \left\{
  \begin{aligned}
    \tfrac{\dd}{\dx[t]}u - \Delta u &\in F(u,v) &&\text{on}\ (0, T)\times \Leb^{2}(\Omega), \\
    \tfrac{\dd}{\dx[t]}v + A(t)v &\in G(u,v) &&\text{on}\ (0, T) \times \mathcal{H}, \\
    (u(0), v(0)) &= (u_0, v_0) &&\text{in}\ \Leb^{2}(\Omega; \mathbb{R}) \times \mathcal{H},
  \end{aligned}
  \right.
\end{equation}
where the second line uses any operator $A$ satisfying the assumptions of \cref{th:LocalExist} and $E = -\Delta$ is the Dirichlet--Laplacian on $H = \Leb^{2}(\Omega)$ for a bounded Lipschitz domain $\Omega\subseteq \mathbb{R}^{d}$. Then $\Delta$ is self-adjoint and non-positive and consequently $\Delta$ is maximal dissipative and generates a $\mathrm{C}_{0}$-semigroup of contractions appealing to the Lumer--Philips theorem. Furthermore, the semigroup has compact resolvent.
\begin{corollary}
  If $F$ and $G$ are bcc, Hausdorff-continuous and half weakly continuous, problem \eqref{eq:Heat} has a unique local solution $(u,v) \in \mathcal{C}\bigl([0,T],\Leb^{2}(\Omega)\bigr)\times \mathcal{C}\bigl([0,T];\mathcal{H}\bigr)$.\\
  If additionally, $F, G$ satisfy a linear growth bound, the solution is global.
\end{corollary}
\begin{proof}
  Apply \cref{th:LocalExist} and \cref{th:GlobalExist}.
\end{proof}
%%%%%%%%%%%%%%%%%%%%%%%%%%%%%%
% S U B S E C T I O N
\subsection{Schr\"odinger semigroup}
\label{subsec:Schroedinger}
The authors of \cite{Simsen2024} studied an autonomous system of inclusions stemming from a generalization of the Schrödinger--Debye system which models the propagation of electromagnetic waves through a nonresonant medium and describes the interaction between a particle and a polarizable medium. In \cite{Simsen2024}, the operator $A\colon \Leb^{2}(\Omega; \mathbb{R})\rightarrow  \Leb^{2}(\Omega; \mathbb{R})$ was a generic maximal monotone operator. We point out that the case $D(t,x)\equiv 1$ in our choice of $A$ (see~\cref{ass:D}) includes the case considered there.

We consider the following nonautonomous Schrödinger-dissipative coupled system of inclusions
\begin{equation}
  \left\{
  \label{eq:Schroedinger}
  \begin{aligned}
    \tfrac{\dd}{\dx[t]}u - \iu \Delta u &\in F(u,v) &&\text{on}\ (0, T) \times \Leb^{2}(\Omega),\\
    \tfrac{\dd}{\dx[t]}v + A(t)v &\in G(u,v) &&\text{on}\ (0, T) \times \Leb^{2}(\Omega),\\
    (u(0), v(0)) &= (u_0, v_0) &&\text{in}\ \Leb^{2}(\Omega) \times \Leb^{2}(\Omega).
  \end{aligned}
  \right.
\end{equation}

The second line of that system contains exactly the specific operator $A$ from \cref{subsec:SchroedingerDebye} satisfying \cref{ass:D} and consequently \cref{ass:A}. For the first line, we observe that $E = -\iu \Delta$, where $\Delta$ is the Dirichlet--Laplacian, is a skew-selfadjoint operator on $H = \Leb^{2}(\Omega)$, because $\Delta$ is self-adjoint. Consequently, by Stone's theorem, the generated semigroup is unitary, in particular it is a $\mathrm{C}_{0}$-semigroup with compact resolvent, since the resolvent set is nonempty and the embedding of the domain into $\Leb^{2}(\Omega)$ is compact.

\begin{corollary}
  If $F$ and $G$ are bcc, Hausdorff-continuous and half weakly continuous, problem \eqref{eq:Schroedinger} has a unique local solution $(u,v) \in \mathcal{C}\bigl([0,T];\Leb^{2}(\Omega)\bigr)\times \mathcal{C}\bigl([0,T];\Leb^{2}(\Omega)\bigr)$.\\
  If additionally, $F, G$ satisfy a linear growth bound, the solution is global.
\end{corollary}
\begin{proof}
  Apply \cref{th:LocalExist} in combination with \cref{th:AssumptionsSatisfied}, as well as \cref{th:GlobalExist}.
\end{proof}

%%%%%%%%%%%%%%%%%%%%%%%%%%%%%%
%    S U B S E C T I O N
\subsection{Wave semigroup}
\label{subsec:Wave}
On a bounded Lipschitz domain $\Omega\subseteq \mathbb{R}^{d}$ we can pose  the wave equation, for some $a^*=a\in \Leb^\infty(\Omega)^{d \times d}$ uniformly strictly positive definite,
\begin{equation*}
  \left\{
    \begin{aligned}
      \tfrac{\dd^{2}}{\dx[t^{2}]}u - \operatorname{div} a \operatorname{grad}_{0}  u &\in F &&\quad\text{on}\ (0, T) \times \Leb^{2}(\Omega; \mathbb{R}),\\
      u(0) &= u_0 &&\quad\text{in}\ \Leb^{2}(\Omega; \mathbb{R}),
    \end{aligned}
  \right.
\end{equation*}
as a system as follows
\begin{equation*}
  \biggl[\partial_{t}\begin{pmatrix} 1 & 0 \\ 0 & a^{-1} \end{pmatrix}
  - \begin{pmatrix} 0 & \operatorname{div} \\ \operatorname{grad}_{0} & 0 \end{pmatrix}\biggr]
  \begin{pmatrix} u_{1} \\ u_{2} \end{pmatrix}
  \in F(u_{1},u_{2},v)\text{;}
\end{equation*}see~\cite[p.~92ff]{Waurick2022} for details. Consequently, we arrive at the formulation
\begin{equation}
  \label{eq:Wave}
  \left\{
  \begin{aligned}
    \biggl[\partial_{t}\!\begin{pmatrix} 1 & 0 \\ 0 & a^{-1} \end{pmatrix}
    \!- \!\begin{pmatrix} 0 & \operatorname{div} \\ \operatorname{grad}_{0} & 0 \end{pmatrix}\!\biggr]
    \!\begin{pmatrix} u_{1} \\ u_{2} \end{pmatrix}
    &\in F(u_{1},u_{2},v) &&\text{on}\ (0, T) \!\times \!H,\\
    \tfrac{\dd}{\dx[t]}v + A(t)v &\in G(u_{1},u_{2},v) &&\text{on}\ (0, T) \!\times \!\mathcal{H},\\
    \big(u_{1}(0), u_{2}(0), v(0)\big) &= (u_{1,0}, u_{2,0}, v_{0}) &&\text{in}\ H \!\times \! \mathcal{H},
  \end{aligned}
  \right.
\end{equation}
where again, $A$ satisfies the assumptions of \cref{th:LocalExist} and $H = \Leb^{2}(\Omega)\times \Leb^{2}(\Omega)^d$. For the operator $E \coloneq \begin{psmallmatrix} 0 & \operatorname{div} \\ \operatorname{grad}_{0} & 0 \end{psmallmatrix}$, where $\operatorname{grad}_{0}$ denotes the weak gradient with Dirichlet boundary condition (i.e., the distributional gradient restricted to the domain $\mathring{\Sob}^1(\Omega)$ given by the completion of smooth compactly supported functions w.r.t.~$\Sob^1$-norm) and $\operatorname{div}=-\operatorname{grad}_0^*$, we can state the domain as
\begin{equation*}
  \operatorname{dom}(E) \coloneq \mathring{\Sob}^{1}(\Omega)\times \Sob(\operatorname{div},\Omega) \subseteq H.
\end{equation*}
$E$ generates a $\mathrm{C}_{0}$-semigroup on $H$, because $E$ is skew-selfadjoint, but it lacks compact resolvent, since the classical Helmholtz-decomposition states
\begin{equation*}
  \Leb^{2}(\Omega) = \nabla \mathring{\Sob}^{1}(\Omega) \oplus \ker (\operatorname{div}).
\end{equation*}
In particular, $\operatorname{curl} \Sob(\operatorname{curl},\Omega) \subseteq \ker (\operatorname{div})$. By factorising the kernel of the divergence, i.e., by transitioning to
\begin{equation*}
  \widetilde{H} = \Leb^{2}(\Omega) \times \nicefrac{\Leb^{2}(\Omega)^{d}}{\ker (\operatorname{div})},
\end{equation*}
and the corresponding generator $\widetilde{E}$ on $\widetilde{H}$, we obtain a compact embedding of
\begin{equation*}
  \operatorname{dom}\bigl(\widetilde{E}\bigr) = \mathring{\Sob}^{1} \times \nicefrac{\Sob(\operatorname{div};\Omega)}{\ker (\operatorname{div})}
\end{equation*}
into $\widetilde{H}$, since $\mathring{\Sob}^{1}(\Omega)\overset{\mathrm{c}}{\hookrightarrow} \Leb^{2}(\Omega)$ by Rellich's theorem and for the quotient space we observe that for any bounded sequence $(f_{n})_{n}$ in $\nicefrac{\Sob(\operatorname{div};\Omega)}{\ker (\operatorname{div})}$, we can write $f_{n}= \operatorname{grad}_{0}g_{n}$ by virtue of the Helmholtz-decomposition. Appealing to Poincar\'e's inequality, the seminorm $\norm{\nabla \argdot}_{2}$ and the norm $\norm{\argdot}_{1,2}$ are equivalent. In particular, $(g_{n})_{n}$ in $\mathring{\Sob}^{1}(\Omega)$ is bounded and we can appeal to Rellich's theorem to obtain a converging subsequence (denoted by the same indices) $g_{n}\to g$ in $\Leb^{2}(\Omega)$. We now observe:
\begin{align*}
  \norm{f_{n}-f_{m}}^{2} &= \dualprod{f_{n}-f_{m}}{\operatorname{grad}_{0}(g_{n}-g_{m})}\\
  &= -\dualprod{\operatorname{div}(f_{n}-f_{m})}{g_{n}-g_{m}},\end{align*}
by definition of the adjoint $(\operatorname{div}^*=-\operatorname{grad}_0)$. Since $(g_{n})_{n}$ in $ \Leb^{2}(\Omega)$ converges and $(\operatorname{div}f_{n})_{n}$ in $ \Leb^{2}(\Omega)$ is bounded, $(f_{n})_{n}$ is a $\Leb^{2}(\Omega)$-Cauchy-sequence and admits a limit.
\begin{corollary}
  If $F$ and $G$ are bcc, Hausdorff-continuous and half weakly continuous, problem \eqref{eq:Wave} posed on $\widetilde{H}\times \mathcal{H}$ has a unique local solution
  \begin{equation*}
    (u_{1},u_{2},v) \in \mathcal{C}\bigl([0,T];\Leb^{2}(\Omega)\bigr)\times \mathcal{C}\bigl([0,T];\nicefrac{\Leb^{2}(\Omega)^{d}}{\ker (\operatorname{div})}\bigr)\times \mathcal{C}\bigl([0,T];\mathcal{H}\bigr).
  \end{equation*}
  If additionally, $F, G$ satisfy a linear growth bound, the solution is global.
\end{corollary}
\begin{proof}
  Apply \cref{th:LocalExist} and \cref{th:GlobalExist}.
\end{proof}

%%%%%%%%%%%%%%%%%%%%%%%%%%%%%%
%    S U B S E C T I O N
\subsection{Maxwell semigroup}
\label{subsec:Maxwell}
On a domain $\Omega\subseteq \mathbb{R}^{3}$ with dielectricity $\epsilon\colon \Omega \to \mathbb{R}^{3\times 3}$, permeability $\mu\colon \Omega \to \mathbb{R}^{3\times 3}$ and (external) current density $j\colon \mathbb{R}\times \Omega \to \mathbb{R}^{3}$ we incorporate Maxwell's equations into our setting by posing
\begin{equation}
  \label{eq:Maxwell}
  \left\{
  \begin{aligned}
    \tfrac{\dd}{\dx[t]}\epsilon u_{1} \!+\! \sigma u_{1} \!-\! \operatorname{curl}u_{2} &\in j(u_{1}) \!+\! F_{1}(u_{1},u_{2},v) &&\text{on}\ (0, T) \!\times\! \Leb^{2}(\Omega)^{3},\\
    \tfrac{\dd}{\dx[t]}\mu u_{2} \!+\! \operatorname{curl}_{0}u_{1} &\in F_{2}(u_{1},u_{2},v) &&\text{on}\ (0, T) \!\times\! \Leb^{2}(\Omega)^{3},\\
    \tfrac{\dd}{\dx[t]}v \!+\! A(\argdot)v &\in G(u_{1},u_{2},v) &&\text{on}\ (0, T) \!\times\! \mathcal{H},\\
    (u_{1}(0), u_{2}(0), v(0)) &= (u_{1,0}, u_{2,0}, v_{0}) &&\text{in}\ \Leb^{2}(\Omega; \mathbb{C})^{3} \!\times\! \Leb^{2}(\Omega; \mathbb{C})^{3} \!\times\! \mathcal{H},
  \end{aligned}
  \right.
\end{equation}
where again, $A$ satisfies the assumptions of \cref{th:LocalExist} and $\operatorname{curl}_{0}$ denotes the weak rotation with Dirichlet boundary (defined similarly as the gradient above, see, again, \cite[p.~92ff]{Waurick2022} for details). We point out that the left-hand side of the system
\begin{equation*}
  \left\{
  \begin{aligned}
    \tfrac{\dd}{\dx[t]}\epsilon u_{1} - \operatorname{curl}u_{2} &\in j(u_{1}) +  F_{1}(u_{1},u_{2},v) &&\text{in}\ (0, T) \times \Leb^{2}(\Omega)^{3}, \\
    \tfrac{\dd}{\dx[t]}\mu u_{2} + \operatorname{curl}_{0}u_{1} &\in F_{2}(u_{1},u_{2},v) &&\text{in}\ (0, T) \times \Leb^{2}(\Omega)^{3},
  \end{aligned}
  \right.
\end{equation*}
can be written as
\begin{equation*}
  \biggl[\partial_{t}
  \begin{pmatrix} \epsilon & 0 \\ 0 & \mu \end{pmatrix}
  + \underbrace{\begin{pmatrix} 0 & - \operatorname{curl} \\ \operatorname{curl}_{0} & 0 \end{pmatrix}}_{\eqcolon E}
  \biggr]
  \begin{pmatrix} u_{1} \\ u_{2} \end{pmatrix}
  \in F(u_{1},u_{2},v).
\end{equation*}
$E$ is a skew-selfadjoint operator on $H\coloneq \Leb^{2}(\Omega)^{3}\times \Leb^{2}(\Omega)^{3}$. Hence,  $E$ generates a $\mathrm{C}_{0}$-semigroup. Note that $E$ does not have compact resolvent, since gradient fields are contained in $\ker (\operatorname{curl})$, i.e., $\nabla \Sob^{1}(\Omega) \subseteq \ker (\operatorname{curl})$ and $\mathring{\nabla}\mathring{\Sob}^{1}(\Omega) \subseteq \ker (\operatorname{curl}_{0})$.\\
By refining the Helmholtz-decomposition from \cref{subsec:Wave}, we obtain (see~\cite[thm.~5.3]{Pauly2016})
\begin{alignat*}{3}
  \Leb^{2}(\Omega)^{3} &= \nabla \Sob^{1}(\Omega) &&\oplus \operatorname{curl}_{0}\Sob(\operatorname{curl}_{0};\Omega) &&\oplus \ker (\operatorname{curl}_{0})\cap \ker (\operatorname{div})\\
  &= \mathring{\nabla} \mathring{\Sob}^{1}(\Omega) &&\oplus \operatorname{curl}\Sob(\operatorname{curl};\Omega) &&\oplus \ker (\operatorname{curl})\cap \ker (\operatorname{div}_{0}),
\end{alignat*}
where the last terms on the right vanish for topologically trivial domains $\Omega$, i.e., $\Omega$ simply connected and $\mathbb{R}^{3}\!\setminus \!\Omega$ connected (see~\cite[rem.~5.4]{Pauly2016}, \cite{PW22,Pi82}). We will assume this case. We can then factorize the null spaces
\begin{equation*}
  \ker (\operatorname{curl}) = \operatorname{grad}_{0} \mathring{\Sob}^{1}(\Omega) \qquad \text{and} \qquad
  \ker (\operatorname{curl}_{0}) = \operatorname{grad} \Sob^{1}(\Omega)
\end{equation*}
to obtain the quotient Hilbert space
\begin{align*}
  \widetilde{H} &\coloneq \nicefrac{\Leb^{2}(\Omega)^{3}}{\operatorname{grad} \Sob^{1}(\Omega)} \times \nicefrac{\Leb^{2}(\Omega)^{3}}{\operatorname{grad}_{0} \mathring{\Sob}^{1}(\Omega)}\\
  \intertext{and the corresponding semigroup generator $\widetilde{E}$ with domain}
  \operatorname{dom}\bigl(\widetilde{E}\bigr) &\coloneq \nicefrac{\Sob(\operatorname{curl}_{0};\Omega)}{\operatorname{grad} \Sob^{1}(\Omega)} \times \nicefrac{\Sob(\operatorname{curl};\Omega)}{\operatorname{grad}_{0} \mathring{\Sob}^{1}(\Omega)}\subseteq \widetilde{H}.
\end{align*}
To verify compactness of the embedding $\operatorname{dom}(\widetilde{E}) \hookrightarrow \widetilde{H}$, one can use that both spaces in the product can be seen as spaces $\Sob(\operatorname{div};\Omega)$ and $\Sob(\operatorname{div}_{0};\Omega)$ with certain tangential trace conditions attached. In fact, this follows from the Picard--Weber--Weck selection theorem, see \cite{Pi84}. For convenience, we also refer to a similar rationale as in \cref{subsec:Wave} in \cite[thm.~4.1]{Skrepek2022}. Hence, one arrives at the following result.
\begin{corollary}
  If $F$ and $G$ are bcc, Hausdorff-continuous, and half weakly continuous and $j\colon \Leb^{2}(\Omega)^{2}\to \Leb^{2}(\Omega)^{3}$ is continuous, problem \eqref{eq:Schroedinger} posed on $\widetilde{H}\times \mathcal{H}$ has a unique local solution
  \begin{equation*}
    (u_{1}, u_{2}, v) \in \mathcal{C}\bigl([0,T];\nicefrac{\Leb^{2}(\Omega)^{3}}{\operatorname{grad} \Sob^{1}(\Omega)}\bigr) \times \mathcal{C}\bigl([0,T]; \nicefrac{\Leb^{2}(\Omega)^{3}}{\operatorname{grad}_{0} \mathring{\Sob}^{1}(\Omega)}\bigr) \times \mathcal{C}\bigl([0,T]; \mathcal{H}\bigr).
  \end{equation*}
  If additionally, $F, G$ satisfy a linear growth bound, the solution is global.
\end{corollary}
\begin{proof}
  Apply \cref{th:LocalExist} and \cref{th:GlobalExist}.
\end{proof}

%%%%%%%%%%%%%%%%%%%%%%%%%%%%%%
%    S E C T I O N
\section*{Acknowledgments}
This work was partially developed when the second author from UNIFEI visited the first and third authors from TUBAF in Freiberg,  Germany. He would like to express his gratitude for their hospitality during his stay. J. Simsen has been partially supported by the Brazilian research agencies FAPEMIG and CNPq, processes FAPEMIG RED-00133-21 and APQ-06590-24 (Bolsista FAPEMIG-CNPq - Brasil).

The first author is funded by a graduate student stipend of the federated German state of Saxony.

%%%%%%%%%%%%%%%%%%%%%%%%%%%%%%
% B I B L I O G R A P H Y
\bibliographystyle{abbrvurl}
\bibliography{references}

\end{document}